\documentclass[reqno]{amsart}
\usepackage{graphicx,float,amssymb}
\usepackage[center]{caption}
\usepackage{tikz}
\usetikzlibrary{calc,positioning,angles}
\tikzset{
sedge/.append style={shorten <=9pt, shorten >=9pt}
}
\makeatletter
\tikzset{
  @arc through/.style 2 args={
    to path={
      \pgfextra
        \pgfextract@process\pgf@tostart{\tikz@scan@one@point\pgfutil@firstofone(\tikztostart)\relax}%
        \pgfextract@process\pgf@tothrough{\tikz@scan@one@point\pgfutil@firstofone#1}%
        \pgfextract@process\pgf@totarget{\tikz@scan@one@point\pgfutil@firstofone(\tikztotarget)\relax}%
        \pgfextract@process\pgf@topointMidA{\pgfpointlineattime{.5}{\pgf@tostart}{\pgf@tothrough}}%
        \pgfextract@process\pgf@topointMidB{\pgfpointlineattime{.5}{\pgf@totarget}{\pgf@tothrough}}%
        \pgfextract@process\pgf@tocenter{%
          \pgfpointintersectionoflines{\pgf@topointMidA}
            {\pgfmathrotatepointaround{\pgf@tothrough}{\pgf@topointMidA}{90}}
            {\pgf@topointMidB}{\pgfmathrotatepointaround{\pgf@tothrough}{\pgf@topointMidB}{90}}}%
        \pgfcoordinate{arc through center}{\pgf@tocenter}%
        \pgfpointdiff{\pgf@tocenter}{\pgf@tostart}%
        \pgfmathveclen@{\pgfmath@tonumber\pgf@x}{\pgfmath@tonumber\pgf@y}%
        \edef\pgf@toradius{\pgfmathresult pt}
        \pgfmathanglebetweenpoints{\pgf@tocenter}{\pgf@tostart}%
        \let\pgf@tostartangle\pgfmathresult
        \pgfmathanglebetweenpoints{\pgf@tocenter}{\pgf@totarget}%
        \let\pgf@toendangle\pgfmathresult
        \ifdim\pgf@tostartangle pt>\pgf@toendangle pt\relax
          \pgfmathsetmacro\pgf@tostartangle{\pgf@tostartangle-360}%
        \fi
        #2%
          \pgfmathsetmacro\pgf@toendangle{\pgf@toendangle-360}%
        \fi
      \endpgfextra
      arc [radius=+\pgf@toradius, start angle=\pgf@tostartangle, end angle=\pgf@toendangle] \tikztonodes
    }},
  arc through ccw/.style={@arc through={#1}{\iffalse}},
  arc through cw/.style={@arc through={#1}{\iftrue}},
}
\makeatother

\makeatletter                                                  %
\def\th@plain{\slshape}                                        %
\makeatother                                                   %

\newcommand{\Cbb}{\mathbb{C}}
\newcommand{\Obb}{\mathbb{O}}
\newcommand{\Qbb}{\mathbb{Q}}
\newcommand{\Rbb}{\mathbb{R}}
\newcommand{\Zbb}{\mathbb{Z}}
\newcommand{\Bcal}{\mathcal{B}}
\newcommand{\Hcal}{\mathcal{H}}
\newcommand{\Rcal}{\mathcal{R}}
\newcommand{\Afrak}{\mathfrak{A}}
\newcommand{\abf}{\mathbf{a}}
\newcommand{\gel}{\gtreqqless}

\newcommand{\m}{^{-1}}

\newcommand{\abs}[1]{\lvert#1\rvert}
\newcommand{\newword}[1]{\textsl{#1}}
\newcommand{\angles}[1]{\langle #1 \rangle}
\newcommand{\floor}[1]{\lfloor #1 \rfloor}
\newcommand{\set}[1]{\{ #1 \}}
\newcommand{\cvect}[2]{\bigl(\begin{smallmatrix}#1\\#2\end{smallmatrix}\bigr)}

\newcommand{\bbmatrix}[4]{\bigl[\begin{smallmatrix}#1&#2\\#3&#4\end{smallmatrix}\bigr]}

\DeclareMathSymbol{\upharpoonright}{\mathrel}{AMSa}{"16}

\DeclareMathOperator{\PP}{P}
\DeclareMathOperator{\PSL}{PSL}

\DeclareMathOperator{\tr}{tr}
\DeclareMathOperator{\Mat}{Mat}
\DeclareMathOperator{\re}{Re}
\DeclareMathOperator{\sgn}{sgn}
\DeclareMathOperator{\CR}{CR}

\theoremstyle{plain}
\newtheorem{theorem}{Theorem}[section]
\newtheorem{lemma}[theorem]{Lemma}

\theoremstyle{definition}
\newtheorem{definition}[theorem]{Definition}
\newtheorem{remark}[theorem]{Remark}
\newtheorem{example}[theorem]{Example}
\newtheorem{notation}[theorem]{Notation}

\begin{document}

\bibliographystyle{plain}

\sloppy

\title[Decreasing height]{Decreasing height\\
along continued fractions}

\author[]{Giovanni Panti}
\address{Department of Mathematics\\
University of Udine\\
via delle Scienze 206\\
33100 Udine, Italy}
\email{giovanni.panti@uniud.it}

\begin{abstract}
The fact that the euclidean algorithm eventually terminates is pervasive in mathematics. In  the language of continued fractions, it can be stated by saying that the orbits of rational points under the Gauss map $x\mapsto x\m-\floor{x\m}$ eventually reach zero. Analogues of this fact for Gauss maps defined over quadratic number fields have relevance in the theory of flows on translation surfaces, and have been established via powerful machinery, ultimately relying on the Veech dichotomy.
In this paper, for each commensurability class of noncocompact triangle groups of quadratic invariant trace field, we construct a Gauss map whose defining matrices generate a group in the class; we then provide a direct and self-contained proof of termination.
As a byproduct, we provide a new proof of the fact that noncocompact triangle groups of quadratic invariant trace field have the projective line over that field as the set of cross-ratios of cusps.

Our proof is based on an analysis of the action of nonnegative matrices with quadratic integer entries on the Weil height of points. As a consequence of the analysis, we show that long symbolic sequences in the alphabet of our maps can be effectively split into blocks of predetermined shape having the property that the height of points which obey the sequence and belong to the base field decreases strictly at each block end.
Since the height cannot decrease infinitely, the termination property follows.
\end{abstract}

\thanks{\emph{2010 Math.~Subj.~Class.}: 11A55; 37P30.}

\maketitle

\section{Introduction}\label{ref1}

The ordinary continued fractions algorithm and its variants (Ceiling fractions, Centered, Odd, Even, Farey, $\ldots$, see~\cite{baladivallee05}) can be seen as the action of a Gauss map on an interval in $\PP^1\Rbb$. Here, by a \newword{Gauss map} we mean a piecewise-projective map~$T$, all of whose pieces have the form $T(x)=\frac{ax+b}{cx+d}$ for some matrix $\bbmatrix{a}{b}{c}{d}$ in the extended modular group $\PSL_2^\pm\Zbb$, \newword{extended} referring to the fact that matrices of determinant $-1$, as well as $+1$, are allowed.
As remarked in~\cite[p.~36]{arnouxhubert00}, this setting has two natural generalizations, obtained by replacing $\PSL_2^\pm\Zbb$ either with $\PSL_n^\pm\Zbb$ (\newword{multidimensional continued fractions}, \cite{schweiger00}, \cite{berthe11}), or with the \newword{extended Hilbert modular group} $\PSL_2^\pm\Obb$ for some totally real ring of algebraic integers $\Obb$; in this paper we are concerned with the latter.

Given a c.f.~algorithm, that we identify with the corresponding map $T$, 
we define~$\Gamma^\pm_T$ as the subgroup of $\PSL_2^\pm\Rbb$ generated by the matrices defining~$T$. Algorithms generating the same group ---or, more generally, commensurable groups--- should be considered close cousins. As maps conjugate by a projective transformation can be identified, commensurability is to be intended in the wide sense: two subgroups $\Gamma,\Gamma'$ of $\PSL_2^\pm\Rbb$ are \newword{commensurable} if there exists $C\in\PSL_2^\pm\Rbb$ such that $C\Gamma C\m\cap\Gamma'$ has finite index both in $C\Gamma C\m$ and in $\Gamma'$.
Since a certain degree of rigidity is required, we have restrictions on the allowed $\Gamma^\pm_T$'s. 
One such restriction is the requirement that any unimodular interval (see~\S\ref{ref2} for definitions) be mapped to any other in at least one and at most two ways, one of them being order-preserving and the other order-reversing.
Other restrictions arise ---although we do not actually use flows in the
present work--- from the desire of preserving the classical connection between c.f.~and flows on the modular surface~\cite{series85}, \cite[Chapter~9]{einsiedlerward}. We then make the standing assumption that 
$\Gamma^\pm_T$ acts on the hyperbolic plane $\Hcal$ in a properly discontinuous way, and that the resulting fuchsian group
$\Gamma_T=\Gamma^\pm_T\cap\PSL_2\Rbb$ be of the first kind.
We will see in Theorem~\ref{ref15} that this assumption yields, as a side effect, that points are uniquely defined by their symbolic orbits.

When $\Gamma_T$ has invariant trace field~$\Qbb$ (see~\S\ref{ref2} for definitions), then it is commensurable with $\PSL_2\Zbb$; this is the arithmetic case, which is well understood. In this paper we take advantage of the recent classification~\cite{nugentvoight17} of hyperbolic triangle groups in terms of their arithmetic dimension (namely, the number of split real places of the associated quaternion algebra) to introduce continued fractions with good properties for every commensurability class of noncocompact triangle groups of arithmetic dimension $2$. In~\S\ref{ref2}, after a brief review of basic concepts and notation, we apply Margulis's rigidity theorem to refine the classification in~\cite{nugentvoight17} to a classification into commensurability classes. It turns out that the set of $16$ noncocompact triangle groups of arithmetic dimension $2$ is split into $4$ subsets according to the invariant trace field
$F_2\in\set{\Qbb(\sqrt2),\Qbb(\sqrt3),\Qbb(\sqrt5),\Qbb(\sqrt6)}$,
and further split into~$9$ commensurability classes.
In~\S\ref{ref3} we introduce a c.f.~algorithm for each of these classes. Our algorithms extend previous work; the algorithm for the class containing the Hecke group $\Delta(2,8,\infty)$ coincides with the octagon Farey map in~\cite{smillieulcigrai11}, \cite{smillieulcigrai10}, which in turn is a ``folded version'' of the map in~\cite{arnouxhubert00}.
In~\cite{caltaschmidt12} continued fractions ---different from ours--- generating the groups $\Delta(3,m,\infty)$ (which are related to the Veech surfaces in~\cite{ward98}, \cite[\S6.1]{bouwmoller10}) are introduced; these triangle groups are $2$-arithmetic precisely for $m=4,5,6$.

Let $T$ be one of our algorithms.
The following is an informal version of the good property referred to above; we will formulate a precise version in Theorem~\ref{ref20}.
\begin{itemize}
\item[(H)] All entries of all matrices defining $T$ are in the integer ring of the invariant trace field $F_2$ of the commensurability class of $\Gamma_T$. The $T$-symbolic orbit of every $\alpha\in F_2$ can be split into 
finite blocks of predetermined shape having the property that the  height of 
the points in the $T$-orbit of~$\alpha$
decreases strictly at each block end. This block-splitting process stops only when the symbolic orbit of a parabolic fixed point is reached.
\end{itemize}
Since by the Northcott theorem~\cite[Theorem~1.6.8]{bombierigubler06} the height of points cannot decrease infinitely, property~(H) implies that the $T$-orbit of any point of $F_2$ must end up in a parabolic fixed point.
This fact is far from granted; in~\cite[Theorem 1]{arnouxschmidt09} it is proved that it cannot hold for the Hecke groups $\Delta(2,m,\infty)$ unless the degree of the invariant trace field is $1$ or $2$ (see also the beginning of~\S\ref{ref5} of the present paper), and it is conjectured that the same degree restriction holds for every Veech group. Note that a finitely generated nonelementary fuchsian group can be realized as a Veech group only if its trace and invariant trace fields coincide~\cite[Lemma~10.1]{hooper13a}.

Property (H) can be seen from various points of view; from the geometric side it implies that all points of $F_2$ are cusps of $\Gamma_T$ (see Theorem~\ref{ref22}), while from the dynamical side it means that all directions of vanishing SAF invariant are parabolic (see~\cite{arnouxschmidt09} for definitions and references). These facts are known in degree $2$, but the proofs are involved. The geometric fact was proved in~\cite[Satz~2]{leutbecher74} for the specific case of the Hecke groups $\Delta(2,m,\infty)$ ($m=5,8,10,12$), and the dynamic fact in~\cite[Appendix]{mcmullen03}, making crucial use of the Veech dichotomy~\cite{veech89}; this is indeed the approach used in~\cite[Theorem~2.3.3]{smillieulcigrai11}. Our proof is direct, algorithmic and self-contained, making no appeal to the Veech dichotomy.

In~\S\ref{ref4} we briefly review the basics of the Weil height. Then, in Theorem~\ref{ref13}, we characterize the matrices
$A\in\PSL_2^\pm\Obb$ ($\Obb$ being the integer ring of a real quadratic field $K$) which increase strictly the height of all points of $K$ contained in a given interval.
In~\S\ref{ref5} we introduce the notion of decreasing block, and use it to formulate property~(H) in a precise way. In the final Theorem~\ref{ref20} we prove that indeed all the algorithms introduced in~\S\ref{ref3} satisfy property~(H).

\subsubsection*{Acknowledgments} The main idea behind Theorem~\ref{ref13} was suggested to me by Umberto Zannier, whom I very sincerely thank. I also thank John Voight for a clarifying discussion on the results in~\cite{nugentvoight17}, and the referee for the detailed reading and valuable suggestions. Some of the pictures and a few computations are done through SageMath.

\section{Triangle groups and commensurability}\label{ref2}

A matrix $A=\bbmatrix{a}{b}{c}{d}\in\PSL^\pm_2\Rbb$ acts on $z\in\PP^1\Cbb$ in the
standard way: $A*z$ equals $\frac{az+b}{cz+d}$ if $\det(A)=1$, and equals 
$\frac{a\bar z+b}{c\bar z+d}$ if $\det(A)=-1$. We use square brackets for matrices to emphasize that we are actually considering their classes in the projective group.
The above action preserves the upper-half plane $\Hcal$ and its boundary $\PP^1\Rbb$.

The \newword{trace field} of a fuchsian group~$\Gamma$ is the field $F_1=\Qbb(\tr[\Gamma])$ generated over the rationals by the traces of the elements of $\Gamma$.
An element $A\in\Gamma$ is \newword{parabolic} if its trace has absolute value~$2$; every $\alpha\in\PP^1\Rbb$ fixed by a parabolic element of~$\Gamma$ is a \newword{cusp} of~$\Gamma$. Two fuchsian groups intersecting in a subgroup of finite index in both have the same set of cusps (of course, the two groups may partition this common set of cusps in group orbits in different ways).
Fuchsian groups of finite covolume are automatically
of the first kind, geometrically finite and finitely generated; see~\cite[Chapter~4]{katok92} for definitions and proofs.

\begin{notation}
We let\label{ref21} $K$ be a totally real number field of degree $d\ge1$, with ring of algebraic integers $\Obb$. We denote by $\Gamma^\pm$ any subgroup of 
the extended Hilbert modular group
$\PSL^\pm_2\Obb$ such that:
\begin{itemize}
\item[(i)] $\Gamma^\pm$ acts on $\Hcal$ in a properly discontinuous way (i.e., for each point $\alpha$ and each compact $C$ the set of elements of $\Gamma^\pm$ mapping $\alpha$ to $C$ is finite);
\item[(ii)] the fuchsian group $\Gamma=\Gamma^\pm\cap\PSL_2\Rbb$ is of finite covolume and has trace field $K$;
\item[(iii)] $0$, $1$, $\infty$ are cusps of $\Gamma$.
\end{itemize}
\end{notation}

A \newword{$\Gamma^\pm$-unimodular interval} is an interval
in $\PP^1\Rbb$ which
is the image $A*I_0$ of the base unimodular interval $I_0=[0,\infty]$ under some element $A$ of $\Gamma^\pm$.
Intervals are ordered ``counterclockwise'' with respect to the natural cyclic order of $\PP^1\Rbb$ (the one obtained by looking at $\PP^1\Rbb$ as the boundary of the Poincar\'e disk).
A \newword{$\Gamma^\pm$-unimodular partition} of $I_0$ is a finite family $I_1,\ldots,I_r$ (with $r\ge2$) of unimodular intervals such that $I_0=\bigcup_{1\le a\le r}I_a$, distinct $I_a$ intersect at most in a common endpoint, and the set of endpoints other than $\infty$ generates~$K$ as a vector space over~$\Qbb$. 
Unimodular partitions are always numbered in accordance with the above order, so that $0\in I_1$, $\infty\in I_r$, and $I_a$ shares a common endpoint with $I_{a+1}$, for $0\le a<r$.
Whenever $\Gamma^\pm$ is clear from the context, we say just unimodular interval and unimodular partition.

\begin{lemma}\label{ref7}
\begin{enumerate}
\item The only element of $\Gamma^\pm$ that fixes $I_0$ setwise and fixes $\set{0,\infty}$ pointwise is the identity. $\Gamma^\pm$ may or may not contain an element that fixes $I_0$ setwise and exchanges $0$ with $\infty$; if it does, then the element is unique and has the form $\bbmatrix{}{\alpha}{\alpha\m}{}$ (blank entries standing for $0$), which acts on $\Hcal$ as a reflection (i.e., M\"obius inversion) through the geodesic connecting $-\alpha$ with $\alpha$.
\item If $\Gamma^\pm$ contains a reflection as above then, for each pair $U,V$ of unimodular intervals, $\Gamma^\pm$ contains precisely two elements mapping $U$ to $V$, one of them being order-preserving and the other order-reversing. If $\Gamma^\pm$ does not contain such a reflection, then each unimodular $U$ can be mapped to each unimodular $V$ by precisely one element of $\Gamma^\pm$.
\end{enumerate}
\end{lemma}
\begin{proof}
Let $A\in\PSL^\pm_2\Rbb$ be a matrix that fixes $I_0$ setwise. If it fixes $0$ and $\infty$, then it is of the form $\bbmatrix{\alpha}{}{}{\alpha\m}$ for some $\alpha>0$. Therefore, if $A$ belongs to $\Gamma^\pm$, then it belongs to $\Gamma$; since $\infty$ is a cusp of $\Gamma$, we have
$\alpha=1$.
On the other hand, if $A$ exchanges $0$ and $\infty$ then it must fix some $0<\alpha<\infty$, and it readily follows that $A=\bbmatrix{}{\alpha}{\alpha\m}{}$. If $\Gamma^\pm$ contains $\bbmatrix{}{\beta}{\beta\m}{}$ as well, then $\Gamma$ contains 
$\bbmatrix{\alpha\beta\m}{}{}{\alpha\m\beta}$ and, by the above, $\alpha=\beta$.
This proves (1), and (2) is an easy consequence.
\end{proof}

We can now make precise our definition of continued fraction algorithm, which we identify with a Gauss map.

\begin{definition}
Let\label{ref8} $K,d,\Obb,\Gamma^\pm$ be as in Notation~\ref{ref21}, and let $A_1,\ldots,A_r\in\Gamma^\pm$ be such that:
\begin{itemize}
\item[(iv)] the family of intervals $I_a=A_a*I_0$, for $a=1,\ldots,r$, is a unimodular partition of $I_0$;
\item[(v)] $\det A_r=+1$.
\end{itemize}
The \newword{slow continued fraction algorithm} determined by $A_1,\ldots,A_r$ is the Gauss map $T:I_0\to I_0$ which is induced on $I_a$ by $A_a\m$ (by convention, the endpoint in $I_a\cap I_{a+1}$ is mapped using $A_a\m$). The maps $A_a\m$ are the \newword{pieces} of $T$.
If $d=2$ (respectively, $d=3$), then we say that $T$ is
\newword{quadratic} (\newword{cubic}).
We denote by $\Gamma^\pm_T$ the subgroup of $\Gamma^\pm$ generated by $A_1,\ldots,A_r$, and we let $\Gamma_T=\Gamma^\pm_T\cap\PSL_2\Rbb$.
\end{definition}

The two restrictions in Definition~\ref{ref8} ---about the determinant of $A_r$ and the definition of $T$ at endpoints--- simplify some of the proofs and are essentially irrelevant. Indeed, ambiguity at endpoints may occur at most once along the $T$-orbit of a point, and is analogous to the existence of two finite ordinary c.f.~expansions for rational numbers.
On the other hand, requiring $\det A_r=+1$ amounts to requiring that $\infty$ be a fixed point for $T$. This can always be achieved, either by replacing $T$ with $T^2$ if $\det A_r=-1=\det A_1$, or by conjugating $T$ with $x\mapsto x\m$ if $\det A_r=-1=-\det A_1$.

As assumed in Notation~\ref{ref21} (and already used in the proof of Lemma~\ref{ref7}), the point~$\infty$ is a cusp of~$\Gamma$, and thus an indifferent fixed point of~$T$.
This forces any $T$-invariant measure absolutely continuous with respect to Lebesgue measure to be infinite, and gives poor approximation properties to the convergents; hence the adjective ``slow'' in Definition~\ref{ref8}. One remedies this by accelerating the algorithm (i.e., considering the first-return map to some subinterval of~$I_0$); note that the acceleration process leaves $\Gamma^\pm_T$ unaltered~\cite[\S2]{panti18}. 
Since we are concerned only with the eventual orbits of points, we do not need any acceleration.

\begin{lemma}
The group $\Gamma_T^\pm$\label{ref11} has finite index in $\Gamma^\pm$, and hence so does $\Gamma_T$ in $\Gamma$.
\end{lemma}
\begin{proof}
By our assumptions $\Gamma^\pm$ has a fundamental domain of finite hyperbolic area; it is then sufficient to show that the same holds for $\Gamma^\pm_T$. 
Let $\Hcal^+=\set{z\in\Hcal:\re z\ge0}$. For every $a\in\set{1,\ldots,r}$ we have
\[
A_a*\Hcal^+\subset\Hcal^+\quad\text{and}\quad
\Hcal={\textstyle\bigcup}\set{A_a^k*\Hcal^+:k\in\Zbb}.
\]
Therefore $\Hcal^+\setminus(A_a*\Hcal^+)$ is a fundamental domain for the infinite cyclic group generated by $A_a$, and this implies that $\Rcal=\Hcal^+\setminus\bigcup\set{A_a*\Hcal^+:1\le a\le r}$ contains a fundamental domain for the action of $\Gamma^\pm_T$.
Let $0=\alpha_0<\alpha_1<\cdots<\alpha_r=\infty$ be the points determining the unimodular partition $I_1,\ldots,I_r$. Then $\Rcal$ is bounded by the full geodesics $g_0,g_1,\ldots,g_r$, where $g_i$ is the geodesic having endpoints $\alpha_{i-1}$ and $\alpha_i\pmod{r+1}$ (see the thickly bordered region in Figure~\ref{fig1} (left)).
By elementary considerations $\Rcal$ has finite hyperbolic area, and our claim follows.
\end{proof}

We note as an aside that in the case $\Gamma^\pm=\PSL^\pm_2\Zbb$ Lemma~\ref{ref11} has a sharper formulation: for every $T$, the index of $\Gamma^\pm_T$ in $\PSL^\pm_2\Zbb$ is bounded by ~$8$~\cite[Theorem~4.1]{panti18}.

We now specialize to the case in which $\Gamma^\pm$ is an extended triangle group.
Let $2\le l\le m\le n\le\infty$ be integers such that $1/l+1/m+1/n<1$; then there exists a triangle in $\Hcal$, unique up to isometry, having angles $\pi/l,\pi/m,\pi/n$. The group generated by reflections in the sides of this triangle is the \newword{extended triangle group} $\Delta^\pm(l,m,n)<\PSL^\pm_2\Rbb$, and its index-$2$ subgroup of orientation-preserving elements is the \newword{triangle group} $\Delta(l,m,n)$, which is a fuchsian group of cofinite volume.
Clearly the original triangle is a fundamental domain for $\Delta^\pm(l,m,n)$; we will call it the \newword{fundamental triangle}.
Let $T$ be a Gauss map obtained from $\Delta^\pm(l,m,n)$ via the construction in Definition~\ref{ref8}. As remarked above, our assumptions imply that~$\Gamma_T$ contains parabolic elements, and thus at least one of the vertices of the fundamental triangle must lie at infinity. Therefore, $n$ equals $\infty$ and $\Delta(l,m,\infty)$ is not cocompact. 

\begin{example}
Consider\label{ref10} the group $\Delta^\pm(4,\infty,\infty)$ and fix as fundamental triangle the one with
vertices $0$, $\infty$, and $\bigl(1+i\tan(\pi/8)\bigr)/2$.
The matrices expressing reflections in the triangle sides are
(see Figure~\ref{fig1} (left))
\begin{equation}\label{eq2}
J=\begin{bmatrix}
-1 & \\
& 1
\end{bmatrix},\quad
P=\begin{bmatrix}
-1 & 1\\
& 1
\end{bmatrix},\quad
M=\begin{bmatrix}
1 & \\
2+\sqrt{2} & -1
\end{bmatrix}.
\end{equation}
Postcomposing $J$ first with $M$, and then with powers of the counterclockwise rotation $MP$ (of order $4$), we obtain matrices $A_1,\ldots,A_7$ as follows:
\begin{gather*}
A_1=MJ=\begin{bmatrix}
1 & 0\\
2+\sqrt{2} & 1
\end{bmatrix},\quad
A_2=(MP)J=\begin{bmatrix}
1 & 1\\
2+\sqrt{2} & 1+\sqrt{2}
\end{bmatrix},\\
A_3=(MP)MJ,\quad
A_4=(MP)^2J,\quad
A_5=(MP)^2MJ,\quad
A_6=(MP)^3J,\\
A_7=(MP)^3MJ=(MP)\m MJ=PJ=\begin{bmatrix}
1 & 1\\
& 1
\end{bmatrix}.
\end{gather*}
The corresponding unimodular partition is 
\[
0<1-\frac{1}{2}\sqrt{2}<-1+\sqrt{2}<\frac{1}{2}<2-\sqrt{2}
<\frac{1}{2}\sqrt{2}<1<\infty.
\]
The matrices $A_1,A_2,A_7$ clearly generate $\Delta^\pm(4,\infty,\infty)$, which then equals $\Gamma_T^\pm$; of course $K=\Qbb(\sqrt{2})$, $d=2$, $\Obb=\Zbb[\sqrt{2}]$.
We draw in Figure~\ref{fig1} (left) the fundamental triangle, its $M$-image (the union of the two is a fundamental domain for $\Delta(4,\infty,\infty)$), and the images of both under successive powers of $MP$.
The image in the unit disk of the
region $\Hcal^+\setminus\bigcup\set{A_j*\Hcal^+:1\le j\le 7}$
of Lemma~\ref{ref11} is bordered by thick geodesics.
We also draw in Figure~\ref{fig1} (right) the resulting map $T$ which equals, up to a conjugation, the octagon Farey map in~\cite{smillieulcigrai11}, \cite{smillieulcigrai10}; indeed, setting $E=\bbmatrix{1-\sqrt{2}}{1}{}{2}$ and $i=1,\ldots,7$, the matrix $\gamma\nu_i$ of~\cite[p.~34]{smillieulcigrai10} is $E\m A_i\m E$.
Note that, since we want $\infty$ to appear as an ordinary point, here and elsewhere
we draw the hyperbolic plane in the disk model via the Cayley map
$2^{-1/2}\bbmatrix{1}{-i}{-i}{1}$.
For the same reason, the graph in Figure~\ref{fig1} (right) is actually the graph of the conjugate of $T$ under the projective transformation
$\bbmatrix{1}{}{1}{1}:[0,\infty]\to [0,1]$.
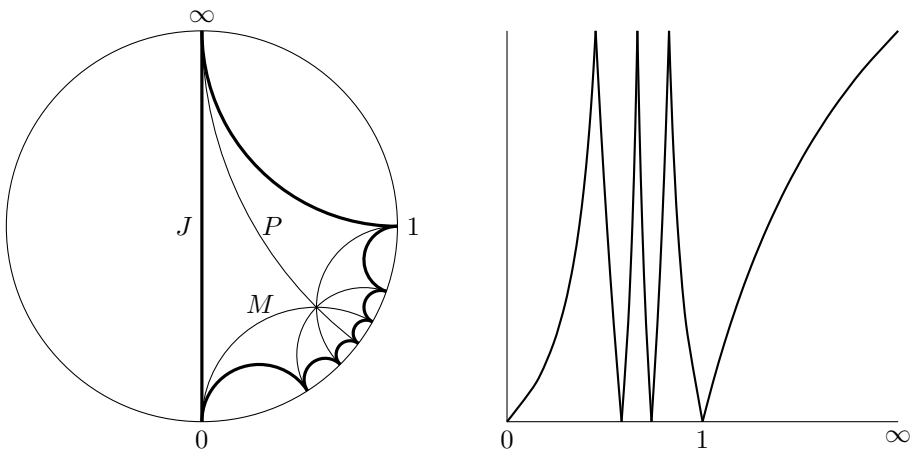
\begin{figure}[h!]
\begin{tikzpicture}[scale=2.6]
\coordinate (p0) at (0,-1);
\coordinate (p2) at (0.7071, -0.7071);
\coordinate (p4) at (0.8722, -0.4890);
\coordinate (p6) at (1, 0);
\coordinate (p1) at (0.5395, -0.8419);
\coordinate (p3) at (0.8, -0.6);
\coordinate (p5) at (0.9428, -0.3333);
\coordinate (p7) at (0, 1);
\coordinate (p8) at (0.5857, -0.4142);
\coordinate (p01) at (0.4361, -0.7445);
\coordinate (p12) at (0.5425, -0.7207);
\coordinate (p23) at (0.7411, -0.5876);
\coordinate (p34) at (0.7734, -0.5469);
\coordinate (p45) at (0.8535, -0.3535);
\coordinate (p56) at (0.8603, -0.27127);
\coordinate (p67) at (0.4, 0.2);
\draw (0,0) circle [radius=1cm];
\draw (p0) to[arc through cw=(p8)] (p4);
\draw (p1) to[arc through cw=(p8)] (p5);
\draw (p2) to[arc through cw=(p8)] (p6);
\draw (p3) to[arc through cw=(p8)] (p7);
\draw[very thick] (p7) to (p0);
\draw[very thick] (p0) to[arc through cw=(p01)] (p1);
\draw[very thick] (p1) to[arc through cw=(p12)] (p2);
\draw[very thick] (p2) to[arc through cw=(p23)] (p3);
\draw[very thick] (p3) to[arc through cw=(p34)] (p4);
\draw[very thick] (p4) to[arc through cw=(p45)] (p5);
\draw[very thick] (p5) to[arc through cw=(p56)] (p6);
\draw[very thick] (p6) to[arc through cw=(p67)] (p7);

\node[left] at (0,0) {$J$};
\node at (0.36,0) {$P$};
\node at (0.3,-0.4) {$M$};

\node[below] at (0,-1) {$0$};
\node[right] at (1,0) {$1$};
\node[above] at (0,1) {$\infty$};

\end{tikzpicture}
\hspace{0.6cm}
\begin{tikzpicture}[scale=2.6]
\draw (-1,1) -- (-1,-1) -- (1,-1);
\draw[thick] plot[smooth] coordinates {(-1,-1) (-0.85090,-0.8) (-0.76230,-0.6) (-0.70359,-0.4) (-0.66183,-0.2) (-0.63060,0) (-0.60636,0.2) (-0.58701,0.4) (-0.57120,0.6) (-0.55804,0.8) (-0.54691,1)};
\draw[thick] plot[smooth] coordinates {(-0.41421,-1) (-0.45418,-0.5) (-0.48904,0) (-0.51971,0.5) (-0.54691,1)};
\draw[thick] plot[smooth] coordinates {(-0.41421,-1) (-0.38071,-0.5) (-0.35924,0) (-0.34431,0.5) (-0.33333,1)};
\draw[thick] plot[smooth] coordinates {(-0.26120,-1) (-0.28909,-0.5) (-0.30839,0) (-0.32252,0.5) (-0.33333,1)};
\draw[thick] plot[smooth] coordinates {(-0.26120,-1) (-0.23103,-0.5) (-0.20710,0) (-0.18767,0.5) (-0.17157,1)};
\draw[thick] plot[smooth] coordinates {(0,-1) (-0.08454,-0.5) (-0.12773,0) (-0.15396,0.5) (-0.17157,1)};
\draw[thick] plot[smooth] coordinates {(0,-1) (0.05263,-0.8) (0.11111,-0.6) (0.17647,-0.4) (0.25,-0.2) (0.33333,0) (0.42857,0.2) (0.53846,0.4) (0.66666,0.6) (0.81818,0.8) (1,1)};
\node[below] at (-1,-1) {$0$};
\node[below] at (0,-1) {$1$};
\node[below] at (1,-1) {$\infty$};
\end{tikzpicture}
\caption{The $(4,\infty,\infty)$ case.\\
Fundamental domain and some images; associated interval map.}
\label{fig1}
\end{figure}

We remark that $\Delta^\pm(4,\infty,\infty)$ does not contain elements of the form 
$\bbmatrix{}{\alpha}{\alpha\m}{}$ (otherwise the fundamental triangle would not be a fundamental domain) and therefore, by Lemma~\ref{ref7}, the pieces of $T$ are uniquely determined by the intervals. In other words, the only matrix in (this specific embedding of) $\Delta^\pm(4,\infty,\infty)$ that maps $I_a$ to $I_0$ is $A_a\m$. Thus, we cannot avoid pieces with negative slopes; if we insist on having all pieces in $\Delta(4,\infty,\infty)$ we must necessarily unfold the map, as in the original construction in~\cite{arnouxhubert00}.
\end{example}

We need more vocabulary.
Let $\Gamma$ be a fuchsian group of finite covolume, and let $\Gamma^2$ be the subgroup generated by the squares of the elements of $\Gamma$; since $\Gamma$ is finitely generated, $\Gamma^2$ is a normal subgroup of finite index.
The \newword{invariant trace field} of $\Gamma$ is the trace field $F_2$ of $\Gamma^2$; of course it is a subfield of the trace field $F_1$ of $\Gamma$. Both $F_2$ and the \newword{invariant quaternion algebra} $F_2[\Gamma^2]$ are invariants (although not complete) of the commensurability class of $\Gamma$~\cite[Theorem~3.3.4 and Corollary~3.3.5]{maclachlanreid03}.

If $\Gamma=\Delta(l,m,n)$ then $F_1$, $F_2$, and $F_2[\Gamma^2]$ can be easily computed. Let $\lambda_l=2\cos(\pi/l)$, and analogously for $m$ and $n$ ($\lambda_\infty=2$); then
\[
F_1=\Qbb(\lambda_l,\lambda_m,\lambda_n),\quad
F_2=\Qbb(\lambda_l^2,\lambda_m^2,\lambda_n^2,\lambda_l\lambda_m\lambda_n).
\]
In the noncocompact case $n=\infty$ ---which is the one concerning us--- the algebra $F_2[\Gamma^2]$ always equals $\Mat_2F_2$ (because
$\Gamma^2$ contains a parabolic element, whose difference with the identity is nilpotent). In the cocompact case $F_2[\Gamma^2]$ is given by the Hilbert symbol~\cite[Lemma~3.5.7]{maclachlanreid03}
\[
\biggl(
\frac{\lambda_m^2(\lambda_m^2-4),\;\lambda_m^2\lambda_n^2
(\lambda_l^2+\lambda_m^2+\lambda_n^2+\lambda_l\lambda_m\lambda_n-4)}{F_2}
\biggr).
\]
Each archimedean place $\sigma:F_2\to\Rbb$ yields a real quaternion algebra $F_2[\Gamma^2]\otimes_\sigma\Rbb$, which necessarily equals either $\Mat_2\Rbb$ (in which case we say that $\sigma$ is \newword{split}) or the Hamilton quaternions. The numbers of archimedean split places is the \newword{arithmetic dimension} of $\Gamma$.

In our noncocompact case, the arithmetic dimension of $\Gamma$ is just the degree of $F_2$ over $\Qbb$, and the list of all triangle groups of arithmetic dimension~$2$ could be generated by exhaustive search. Indeed, if $\Delta(l,m,\infty)$ is such a group, then its invariant trace field must contain $\cos(2\pi/l)$, which has degree~$\varphi(l)/2$ over $\Qbb$. Thus we have the bound $\varphi(l)\le4$, 
yielding $l\in\set{2,3,4,5,6,8,10,12}$,
and similarly for~$m$ if different from~$\infty$. We spare such search by 
appealing to~\cite{nugentvoight17}, which proves that the set of triangle groups, including the cocompact ones, of bounded arithmetic dimension is finite, and provides efficient algorithms for its enumeration.
In particular, \cite[\S6]{nugentvoight17} gives explicit lists of all triangle groups of low dimension: there are $76$ cocompact and $9$ noncocompact groups of dimension $1$ (this is the arithmetic case, already settled in~\cite{takeuchi77a}) and, for dimension from $2$ up to $5$, the corresponding pairs of cardinalities are $(148,16)$, $(111,13)$, $(286,31)$, $(94,6)$. 

In~\cite[Proposition~1]{takeuchi77b} Takeuchi proves that in the arithmetic case the pair $(F_2,F_2[\Gamma^2])$ is a complete invariant of the commensurability class of $\Gamma$, and uses this fact to refine his own work in~\cite{takeuchi77a} to a classification of all $85$ triangle groups of arithmetic dimension~$1$ up to commensurability.
The only noncocompact commensurability class is that of the modular group $\PSL_2\Zbb=\Delta(2,3,\infty)$, and it turns out that there are~$18$ cocompact classes. When the arithmetic dimension is greater than $1$ then the pair $(F_2,F_2[\Gamma^2])$ is not complete anymore, so we apply a different strategy.

\begin{theorem}
Two\label{ref9} nonarithmetic triangle groups $\Gamma,\Delta$ are commensurable if and only if there exists $C\in\PSL_2\Rbb$ and a triangle group $\Xi$ containing both $\Gamma$ and $C\Delta C\m$ with finite index.
\end{theorem}

By the results in~\cite{greenberg63}, \cite{singerman72}, the relation of inclusion among triangle groups is easily settled, so we can split
the classes in~\cite[\S6]{nugentvoight17} in commensurability subclasses. In particular, the~$16$ noncocompact triangle groups of arithmetic dimension~$2$
are split according to the table in Figure~\ref{fig2}. In it, the invariant trace field is displayed to the left, and a group lies over another if and only if the first one is a subgroup of the second (always of index $2$). Here and later on, for short and when no confusion may arise, we tacitly suppress the initial $\Delta$ (or $\Delta^\pm$) from our notation for triangle groups.
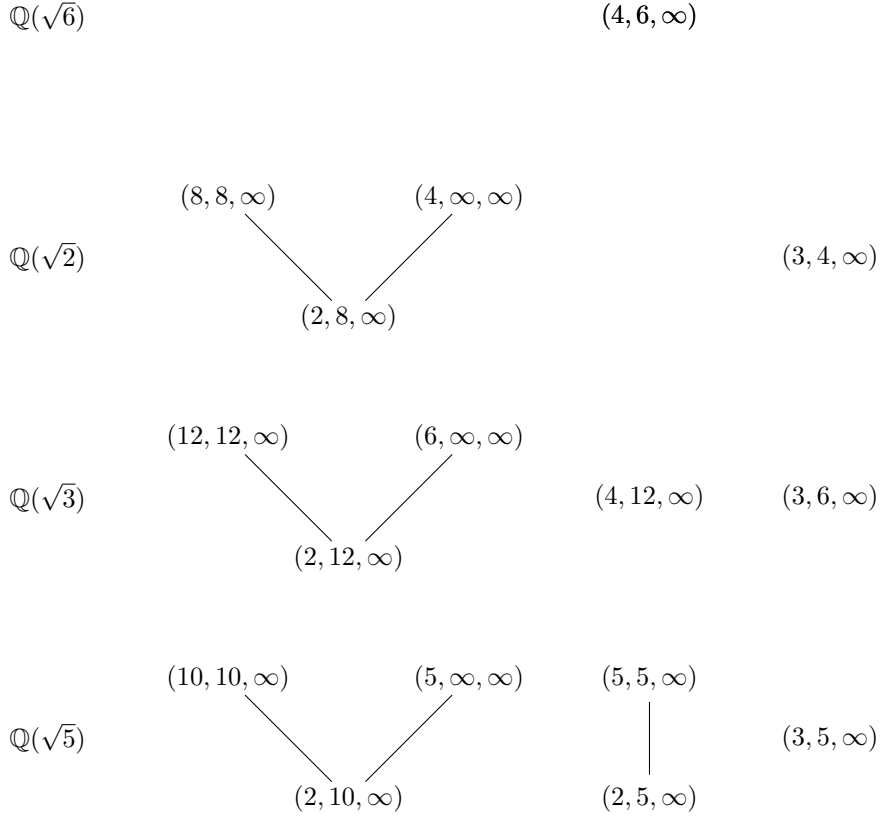
\begin{figure}[h!]
\begin{tikzpicture}[scale=0.8]
\node at (0,0) []  {$\Qbb(\sqrt{5})$};
\node at (0,4) []  {$\Qbb(\sqrt{3})$};
\node at (0,8) []  {$\Qbb(\sqrt{2})$};
\node at (0,12) []  {$\Qbb(\sqrt{6})$};
\node at (3,1) []  {$(10,10,\infty)$};
\node at (3,5) []  {$(12,12,\infty)$};
\node at (3,9) []  {$(8,8,\infty)$};
\node at (5,-1) []  {$(2,10,\infty)$};
\node at (5,3) []  {$(2,12,\infty)$};
\node at (5,7) []  {$(2,8,\infty)$};
\node at (7,1) []  {$(5,\infty,\infty)$};
\node at (7,5) []  {$(6,\infty,\infty)$};
\node at (7,9) []  {$(4,\infty,\infty)$};

\node at (10,-1) []  {$(2,5,\infty)$};
\node at (10,1) []  {$(5,5,\infty)$};
\node at (10,4) []  {$(4,12,\infty)$};
\node at (10,12) []  {$(4,6,\infty)$};
\node at (10,12) []  {$(4,6,\infty)$};

\node at (13,0) []  {$(3,5,\infty)$};
\node at (13,4) []  {$(3,6,\infty)$};
\node at (13,8) []  {$(3,4,\infty)$};

\path (5,-1) edge[sedge] (3,1);
\path (5,3) edge[sedge] (3,5);
\path (5,7) edge[sedge] (3,9);
\path (5,-1) edge[sedge] (7,1);
\path (5,3) edge[sedge] (7,5);
\path (5,7) edge[sedge] (7,9);

\path (10,-1) edge[sedge] (10,1);
\end{tikzpicture}
\caption{Noncocompact commensurability classes in dimension~$2$.}
\label{fig2}
\end{figure}

\begin{proof}[Proof of Theorem~\ref{ref9}]
A proof is sketched in~\cite[p.~759]{girondo_et_al12}; we supply a more detailed version for the reader's convenience. 
For the nontrivial direction, assume that $\Gamma$ and $\Delta$ are commensurable.
Then, after replacing $\Delta$ with an appropriate $C\Delta C\m$, we assume that $\Gamma$ and $\Delta$ intersect in a subgroup of finite index in both. Therefore $\Gamma$ and $\Delta$ have the same \newword{commensurator} $\Xi$, the latter being the group of all $D\in\PSL_2\Rbb$ such that $\Gamma$ and $D\Gamma D\m$ (respectively, $\Delta$ and $D\Delta D\m$) intersect in a subgroup of finite index in both. Since we are in the nonarithmetic case, the Margulis dichotomy~\cite{margulis75} implies that both $\Gamma$ and $\Delta$ have finite index in $\Xi$; hence $\Xi$ is fuchsian as well. But it is known that the only fuchsian groups that can contain a triangle group with finite index are themselves triangle groups (because the Teichm\"uller space of the larger group embeds in the space of the subgroup, and triangle groups are precisely those noncyclic fuchsian groups having trivial spaces; see~\cite[\S6]{singerman72}).
\end{proof}

\section{Continued fraction algorithms}\label{ref3}

In this section we introduce, for each commensurability class in Figure~\ref{fig2}, a Gauss map $T$ as in Definition~\ref{ref8} such that:
\begin{enumerate}
\item $T$ is continuous (equivalently, $\det A_a=(-1)^{r-a}$ for every $a=1,\ldots,r$);
\item $K$ is the invariant trace field of the class;
\item $\Gamma_T$ is a group in the class, with the exception of the
cases $(4,6,\infty)$ and $(4,12,\infty)$, in which $\Gamma_T$ has index~$2$ in the unique group of the class;
\item property (H) holds.
\end{enumerate}
Note that the exception in (3) is unavoidable, since the invariant trace field of $\Delta(4,6,\infty)$ and of $\Delta(4,12,\infty)$ is strictly contained in the respective trace fields.

The integer rings of our invariant trace fields $F_2=K=\Qbb(\sqrt{d})$, for $d=2,3,5,6$, are $\Zbb[\sqrt{2}],\Zbb[\sqrt{3}],\Zbb[\tau],\Zbb[\sqrt{6}]$, respectively ($\tau=(1+\sqrt{5})/2$ is the golden ratio). All of them have class number $1$; since the class group is in natural correspondence with the $\PSL_2\Obb$-orbits in $\PP^1K$,
this is surely necessary for all points in $K$ to be cusps of $\Gamma_T$ and hence, as remarked in~\S\ref{ref1}, for property (H) to hold.
Note that each commensurability class in Figure~\ref{fig2} contains a group whose cusps form a single group orbit.

\subsection{Case $(l,m,\infty)\in\set{(2,5,\infty),(3,4,\infty),(3,5,\infty),(3,6,\infty)}$} This\label{ref17} is the simplest and most flexible case: we take as fundamental triangle the one having vertices $-\exp(-i\pi/l)$, $\exp(i\pi/m)$, and $\infty$.
Let
\begin{equation*}
F=\begin{bmatrix}
 & 1\\
1 & 
\end{bmatrix},
\quad
Q_l=\begin{bmatrix}
\lambda_l & 1\\
-1 &
\end{bmatrix},\quad
R_m=\begin{bmatrix}
 & 1\\
-1 & \lambda_m
\end{bmatrix}.
\end{equation*}
Then $Q_l$ is a counterclockwise rotation of order $l$ around $-\exp(-i\pi/l)$, and $R_m$ is a counterclockwise rotation of order $m$ around $\exp(i\pi/m)$. Also, $\Delta(l,m,\infty)$ is a free product $\angles{Q_l}*\angles{R_m}=Z_l*Z_m$ of cyclic groups, while 
$\Delta^\pm(l,m,\infty)$ is an amalgamated product $\angles{Q_l,F}*_{\angles{F}}\angles{R_m,F}=D_l*_{Z_2}D_m$
of dihedral groups. 
Let $T'$ be the Gauss map determined by the $(l-1)(m-1)$ matrices $A'_{ij}=R_m^iQ_l^j$, for $1\le i<m$ and $1\le j<l$.
Figure~\ref{fig3} (left), which is drawn for $(l,m,\infty)=(3,4,\infty)$,
clarifies the definition; appropriate powers of $Q_l$ map $I_0$ to  subintervals of $[\infty,0]$, and powers of $R_m$ map back these subintervals inside $I_0$ (for short, in Figure~\ref{fig3} we dropped the subscripts in $Q_3$, $R_4$).
All matrices $A'_{ij}$ have positive determinant, so $T'$ is not continuous and has all pieces of positive slope. Replacing any $A'_{ij}$ with $A'_{ij}F$ we have full flexibility in converting each slope from positive to negative. In order to obtain a continuous map we perform this matrix replacement precisely at those intervals $I_1,\ldots,I_r$ ($r=(l-1)(m-1)$) having odd index; this completely defines our final map $T$. 
It is easy to check that the resulting matrices $A_1,\ldots,A_r$ generate $\Delta^\pm(l,m,\infty)$, and therefore $\Gamma_T=\Delta(l,m,\infty)$. Since here $\lambda_l$ equals $0$ or $1$, for this group the trace field agrees with the invariant trace field. The required properties (1), (2), (3) are thus satisfied; we will prove (4) in~\S\ref{ref5}.
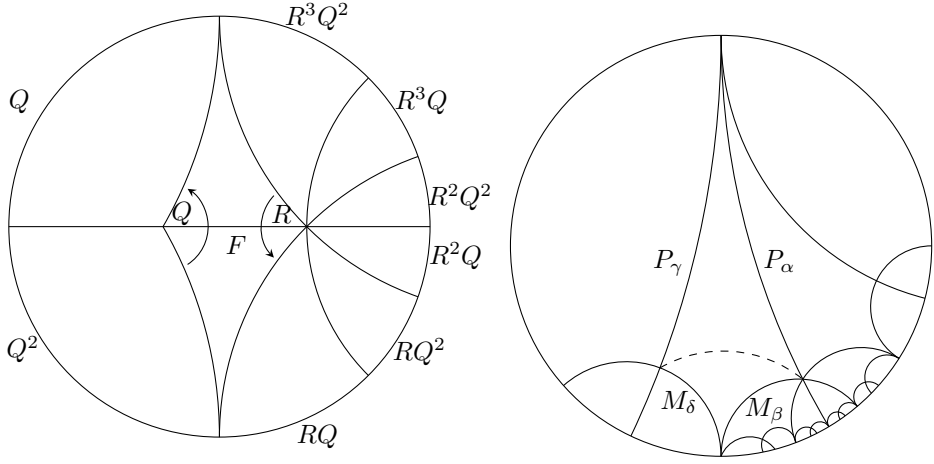
\begin{figure}[h!]
\begin{tikzpicture}[scale=2.8]
\coordinate (p0) at (0,-1);
\coordinate (p1) at (0.7071, -0.7071);
\coordinate (p2) at (0.9428, -0.3333);
\coordinate (p3) at (1, 0);
\coordinate (p4) at (0.9428, 0.3333);
\coordinate (p5) at (0.7071, 0.7071);
\coordinate (p6) at (0, 1);
\coordinate (p7) at (-1, 0);
\coordinate (p8) at (-0.2679491, 0);
\coordinate (p9) at (0.41421, 0);
\coordinate (p10) at (-0.8, -0.6);
\coordinate (p11) at (-0.8, 0.6);
\draw (0,0) circle [radius=1cm];
\draw (p7) to (p3);
\draw (p8) to[arc through ccw=(p10)] (p6);
\draw (p0) to[arc through ccw=(p11)] (p8);
\draw (p0) to[arc through cw=(p9)] (p4);
\draw (p1) to[arc through cw=(p9)] (p5);
\draw (p2) to[arc through cw=(p9)] (p6);
\node[below] at (0.08,0) {$F$};
\node at ($(p8)+(0.09,0.06)$) {$Q$};
\node at ($(p9)+(-0.12,0.06)$) {$R$};
\coordinate (pp) at (120:1);
\coordinate (qq) at (-120:1);
\pic[draw,-stealth,angle radius=0.6cm,angle eccentricity=1]{angle = pp--p9--qq};
\coordinate (pp1) at (-70:1);
\coordinate (qq1) at (70:1);
\pic[draw,-stealth,angle radius=0.6cm,angle eccentricity=1]{angle = pp1--p8--qq1};
\node at (-65:1.1) {$RQ$};
\node at (65:1.1) {$R^3Q^2$};
\node at (-32:1.12) {$RQ^2$};
\node at (212:1.1) {$Q^2$};
\node at (32:1.13) {$R^3Q$};
\node at (148:1.12) {$Q$};
\node at (-7:1.13) {$R^2Q$};
\node at (7:1.15) {$R^2Q^2$};
\end{tikzpicture}
\begin{tikzpicture}[scale=2.8]
\coordinate (p0) at (0,-1);
\coordinate (p1) at (0.2,-0.9797);
\coordinate (p2) at (0.2562,-0.9666);
\coordinate (p3) at (0.3550,-0.9348);
\coordinate (p4) at (0.4278,-0.9038);
\coordinate (p5) at (0.4545,-0.8907);
\coordinate (p6) at (0.5117,-0.8591);
\coordinate (p7) at (0.5657,-0.8245);
\coordinate (p8) at (0.5890,-0.8081);
\coordinate (p9) at (0.6468,-0.76258);
\coordinate (p10) at (0.71428,-0.6998);
\coordinate (p11) at (0.74787,-0.6638);
\coordinate (p12) at (0.8449,-0.5348);
\coordinate (p13) at (0.9684,-0.2492);
\coordinate (p14) at (1,0);
\coordinate (p15) at (0,1);
\coordinate (p16) at (-0.7478,-0.66383);
\coordinate (p17) at (-0.42787,-0.9038);
\coordinate (q0) at (-0.2898,-0.5797);
\coordinate (q1) at (0.1963,-0.9195);
\coordinate (q2) at (0.4269,-0.86518);
\coordinate (q3) at (0.5572,-0.7894);
\coordinate (q4) at (0.68516,-0.6477);
\coordinate (q5) at (0.7379,-0.1658);
\coordinate (q6) at (0.38799,-0.633593);
\coordinate (rr) at (0,-0.5);
\draw (0,0) circle [radius=1cm];
\draw (p16) to[arc through cw=(q0)] (p0);
\draw (p15) to[arc through cw=(q0)] (p17);
\draw (p0) to[arc through cw=(q6)] (p9);
\draw (p3) to[arc through cw=(q6)] (p12);
\draw (p6) to[arc through cw=(q6)] (p15);
\draw (p0) to[arc through cw=(q1)] (p2);
\draw (p1) to[arc through cw=(q1)] (p3);
\draw (p3) to[arc through cw=(q2)] (p5);
\draw (p4) to[arc through cw=(q2)] (p6);
\draw (p6) to[arc through cw=(q3)] (p8);
\draw (p7) to[arc through cw=(q3)] (p9);
\draw (p9) to[arc through cw=(q4)] (p11);
\draw (p10) to[arc through cw=(q4)] (p12);
\draw (p12) to[arc through cw=(q5)] (p14);
\draw (p13) to[arc through cw=(q5)] (p15);
\draw [dashed] (q0) to[arc through cw=(rr)]  (q6);
\node at ($(q0)+(0.04,0.5)$) {$P_\gamma$};
\node at ($(q0)+(0.09,-0.16)$) {$M_\delta$};
\node at ($(q6)+(-0.11,0.56)$) {$P_\alpha$};
\node at ($(q6)+(-0.18,-0.15)$) {$M_\beta$};
\end{tikzpicture}
\caption{Fundamental domain and some images in the $(3,4,\infty)$ and $(4,6,\infty)$ cases.}
\label{fig3}
\end{figure}

\subsection{Cases $(l,\infty,\infty)$, for $l=4,5,6$}\label{ref6}
The case $(4,\infty,\infty)$ has already been dealt with in Example~\ref{ref10}. In order to treat all cases in a uniform way, we first prove a lemma.

\begin{lemma}
Let\label{ref12} $\alpha,\beta\in\Rbb^*$, and let 
\[
P_\alpha=\begin{bmatrix}
-1 & \alpha\\
& 1
\end{bmatrix},
\qquad
M_\beta=\begin{bmatrix}
1 & \\
\beta & -1
\end{bmatrix}.
\]
Then these matrices express the reflections of $\Hcal$ in the geodesic $g_\alpha$ connecting $\infty$ with $\alpha/2$, and the geodesic $g_\beta$ connecting $0$ with $2/\beta$, respectively. Let $g_0$ be the geodesic connecting $0$ with $\infty$, which is the axis of reflection of the matrix $J$ in Example~\ref{ref10}.
The three geodesics $g_\alpha$, $g_\beta$, and $g_0$, form a hyperbolic triangle with
two vertices at infinity and another ---the one at which $g_\alpha$ and $g_\beta$
meet--- inside~$\Hcal$. Let us call the angle at this third vertex the \newword{inner angle}, of measure $2\pi/p$ (we still let $\lambda_p=2\cos(\pi/p)$, even though here~$p$ might be a non-integer). We then have  
$2\pi/p\le\pi/2$ if and only if $2\le\alpha\beta<4$; if this happens, then $\alpha\beta=\lambda_p^2=2+\lambda_{p/2}$, and the bisectrix of the inner angle is the geodesic connecting $-\sqrt{\alpha/\beta}$ with $\sqrt{\alpha/\beta}$.
\end{lemma}
\begin{proof}
The first statement is easily checked, and implies that $g_\alpha,g_\beta$ intersect iff (either $0<\alpha/2<2/\beta$ or $2/\beta<\alpha/2<0$) iff $0<\alpha\beta<4$. Assume that they intersect with an inner angle of $2\pi/p$, for some $2<p$. Then, by the duplication formula for cosines,
$\abs{4\cos^2(\pi/p)-2}=\abs{2\cos(2\pi/p)}=\abs{\tr(M_\beta P_\alpha)}=\abs{\alpha\beta-2}$, so that
\begin{itemize}
\item either $\lambda_p^2-2=\alpha\beta-2$, i.e., $2/\beta=2\alpha/\lambda_p^2$,
\item or $\lambda_p^2-2=-\alpha\beta+2$, i.e., $2/\beta=2\alpha/(4-\lambda_p^2)$.
\end{itemize}
These two cases correspond to the sign ambiguity in $\cos(2\pi/p)$, namely, to the inner angle being acute or obtuse. Since we want it to be acute, we must choose $2/\beta$ to be smaller in absolute value. As $2\pi/p\le\pi/2$ implies $\lambda_p^2\ge2$, the second case is ruled out and
$\alpha\beta=\lambda_p^2$. The last claim follows by solving for $x$ the equation $P_\alpha=\bbmatrix{}{x}{x\m}{}M_\beta\bbmatrix{}{x}{x\m}{}\m$.
\end{proof}

Generalizing Example~\ref{ref10}, we treat case $(l,\infty,\infty)$ by using a fundamental triangle bounded by $g_\alpha,g_\beta,g_0$, having an inner angle of $\pi/l$, and such that $J,P_\alpha,M_\beta\in\PSL_2^\pm\Obb$, where $\Obb$ is the appropriate integer ring. According to Lemma~\ref{ref12}, the ways of accomplishing this are in 1--1 correspondence with the factorization of $2+\lambda_l$ in~$\Obb$; since $\Obb$ is a unique factorization domain, there is just one way, up to conjugation by units.

For $l=4$, the number $2+\lambda_4=2+\sqrt{2}$ is prime in $\Zbb[\sqrt{2}]$, and the latter has $1+\sqrt{2}$ as fundamental unit. Hence the admissible pairs $\alpha,\beta$ are exactly, up to exchanging $\alpha$ with $\beta$,
\[
\alpha=(1+\sqrt{2})^k,\quad
\beta=(1+\sqrt{2})^{-k}(2+\sqrt{2}),
\]
for $k\in\Zbb$. In Example~\ref{ref10} we just set $k=0$. The choice $k=1$ would have been slightly more symmetric, as $\alpha/\beta$ is then nearer to~$1$, but would have made the construction in~\S\ref{ref5} more involved; hence we maintain $k=0$ for our official~$T$.

For $l=5$, $2+\lambda_5=2+\tau$ is prime in $\Zbb[\tau]$ while, for $l=6$, $2+\lambda_6=2+\sqrt{3}$ is the fundamental unit in $\Zbb[\sqrt{3}]$. We retain $\alpha=1$ in both cases, and set $\beta$ equal to either $2+\tau$ or $2+\sqrt{3}$. Writing $P$ for $P_1$, the rotation $M_\beta P$ has then
order~$l$, and $T$ has $2l-1$ pieces, with $A_1,\ldots,A_{2l-1}$ given by
\[
A_{1+2q}=(M_\beta P)^qM_\beta J,\quad A_{2+2q}=(M_\beta P)^{q+1}J,\quad
A_{2l-1}=(M_\beta P)^{l-1}M_\beta J=\begin{bmatrix}
1 & 1\\
& 1
\end{bmatrix},
\]
for $0\le q\le l-2$.
We still have $\Gamma^\pm_T=\Delta^\pm(l,\infty,\infty)$ and (1), (2), (3) are satisfied. Here the situation is much more rigid than in~\S\ref{ref17}. Indeed $\Delta^\pm(l,\infty,\infty)$ does not contain any reflection of the form given in Lemma~\ref{ref7}, and no flexibility is allowed in the orientation of the pieces of $T$.

\subsection{Cases $(4,m,\infty)$, for $m=6,12$}\label{ref18}
In both cases the trace field is $\Qbb(\sqrt{2},\sqrt{3})$, which contains strictly the invariant trace fields $\Qbb(\sqrt{6})$ and
$\Qbb(\sqrt{3})$ (for $m=6$ and $m=12$, respectively); therefore our $\Gamma^\pm_T$ has to be a proper subgroup of $\Delta^\pm(4,m,\infty)$. We will use a reflection group, with fundamental domain the union of two $(4,m,\infty)$-triangles, images of each other under a reflection in a common edge (the dashed geodesic arc in Figure~\ref{fig3} (right)).

Let us first treat the case $m=6$; according to Lemma~\ref{ref12}, we must find $\alpha,\beta,\gamma,\delta\in\Zbb[\sqrt{6}]$ such that
\begin{align}\label{eq1}
\alpha\beta&=\lambda_6^2=3,\quad \alpha,\beta>0,\nonumber\\
\gamma\delta&=\lambda_4^2=2,\quad \gamma,\delta<0,\\
\alpha/\beta&=\gamma/\delta\nonumber.
\end{align}
In $\Zbb[\sqrt{6}]$ we have the fundamental unit $\varepsilon=5+2\sqrt{6}$, and the unique factorizations $3=\varepsilon(3-\sqrt{6})^2$, 
$2=\varepsilon\m(-2-\sqrt{6})^2$. The equations~\eqref{eq1} imply $6=(\alpha\delta)^2$, so that $\alpha\delta=-\sqrt{6}=(3-\sqrt{6})(-2-\sqrt{6})$. Therefore the general solution of~\eqref{eq1} is
\begin{align*}
\alpha&=\varepsilon^k(3-\sqrt{6}),& \beta&=\varepsilon^{-k+1}(3-\sqrt{6}),\\
\gamma&=\varepsilon^{k-1}(-2-\sqrt{6}),& \delta&=\varepsilon^{-k}(-2-\sqrt{6}).
\end{align*}
Looking for $\alpha/\beta=\varepsilon^{2k-1}$ to approach $1$ as much as possible, we take $k=0$ and set
\begin{align*}
P_\alpha&=\begin{bmatrix}
-1 & 3-\sqrt{6}\\
&1
\end{bmatrix},
&M_\beta&=\begin{bmatrix}
1 & \\
3+\sqrt{6}&-1
\end{bmatrix},\\
P_\gamma&=\begin{bmatrix}
-1 & 2-\sqrt{6}\\
&1
\end{bmatrix},
&M_\delta&=\begin{bmatrix}
1 & \\
-2-\sqrt{6}&-1
\end{bmatrix}.
\end{align*}

Now consider the matrices
\begin{align*}
U_1&=P_\gamma,&U_2&=P_\gamma M_\delta,&U_3&=P_\gamma M_\delta P_\gamma,\\
V_1&=M_\beta,&V_2&=M_\beta P_\alpha,&V_3&=(M_\beta P_\alpha)M_\beta,\\
V_4&=(M_\beta P_\alpha)^2,&V_5&=(M_\beta P_\alpha)^2M_\beta.
\end{align*}
As in~\S\ref{ref17}, the $U_i$'s map $I_0$ to subintervals of $[\infty,0]$, and the $V_j$'s map back these subintervals inside $I_0$. Few minutes of reflection show that the matrices
\begin{align*}
A_1&=V_1U_3,&A_2&=V_1U_2,&A_3&=V_1U_1,\\
A_4&=V_2U_1,&A_5&=V_2U_2,&A_6&=V_2U_3,\\
A_7&=V_3U_3,&A_8&=V_3U_2,&A_9&=V_3U_1,\\
A_{10}&=V_4U_1,&A_{11}&=V_4U_2,&A_{12}&=V_4U_3,\\
A_{13}&=V_5U_3,&A_{14}&=V_5U_2,&A_{15}&=V_5U_1,
\end{align*}
---which have determinant $\pm1$, alternatively--- determine a unimodular partition of $I_0$, whose $15$ intervals are displayed along the right border of the disk model in Figure~\ref{ref3} (right), starting with $I_1=A_1*I_0=[0,5-2\sqrt{6}]$, $I_2=A_2*I_0=[5-2\sqrt{6},(8-3\sqrt{6})/5]$, and ending with $I_{15}=A_{15}*I_0=[1,\infty]$. 
As a consequence, the matrices $A_1,\ldots,A_{15}$ determine a Gauss map $T$, and one easily checks that $\Gamma_T^\pm$ equals $\angles{P_\alpha,M_\beta,P_\gamma,M_\delta}$, which is an index-$2$ subgroup of $\Delta^\pm(4,6,\infty)$. The conditions (1), (2), (3) are thus satisfied. The fact that $\Gamma_T^\pm$ does not contain a reflection as in Lemma~\ref{ref7} corresponds to the fact that $\alpha/\beta=5-2\sqrt{6}$ is not a square in $\Zbb[\sqrt{6}]$, but becomes a square ---namely, $(\sqrt{3}-\sqrt{2})^2$--- in the integer ring of $\Qbb(\sqrt{2},\sqrt{3})$.

The case $(4,12,\infty)$ is analogous; we solve~\eqref{eq1} (with $\lambda_6^2$ replaced by $\lambda_{12}^2=2+\sqrt{3}$) for $\alpha,\beta,\gamma,\delta\in\Zbb[\sqrt{3}]$. Now $2+\sqrt{3}$ is the fundamental unit $\varepsilon$ of $\Zbb[\sqrt{3}]$, while $2=\varepsilon(1-\sqrt{3})^2$. Arguing as above, the general solution to~\eqref{eq1} is $\alpha=\varepsilon^k$, $\beta=\varepsilon^{-k+1}$, $\gamma=\varepsilon^k(1-\sqrt{3})$, $\delta=\varepsilon^{-k+1}(1-\sqrt{3})$. We take $k=0$ and obtain the four matrices
\begin{align*}
P_\alpha&=\begin{bmatrix}
-1 & 1\\
&1
\end{bmatrix},
&M_\beta&=\begin{bmatrix}
1 & \\
2+\sqrt{3}&-1
\end{bmatrix},\\
P_\gamma&=\begin{bmatrix}
-1 & 1-\sqrt{3}\\
&1
\end{bmatrix},
&M_\delta&=\begin{bmatrix}
1 & \\
-1-\sqrt{3}&-1
\end{bmatrix}.
\end{align*}
The matrices $U_1,U_2,U_3$ are defined precisely as in the previous case, while the list of the $V_j$'s is extended by $V_6=(M_\beta P_\alpha)^3$, $V_7=(M_\beta P_\alpha)^3M_\beta$, $\ldots$, $V_{11}=(M_\beta P_\alpha)^5M_\beta$. The products $V_jU_i$, in full analogy with the above, yield matrices $A_1,\ldots,A_{33}$ and a corresponding Gauss map, for which all of the above considerations hold verbatim.

\section{The height on the projective line}\label{ref4}

We recall the definition and the basic properties of the height. Let $K$ be an extension of $\Qbb$ of degree $d$. For each place $P$ of $K$ we denote by 
$\abs{\phantom{a}}_P$ the valuation in~$P$ normalized by $\abs{\alpha}_P=\abs{\sigma(\alpha)}^{e/d}$ (if $P$ is archimedean associated to the embedding $\sigma:K\to\Cbb$, with $e=1$ or $2$ according whether $\sigma$ is real or complex), and by $\abs{p}_P=p^{-d(P|p)/d}$ (if $P$ is nonarchimedean lying over the rational prime $p$, with local degree $d(P|p)=\deg[K_P:\Qbb_p]$). Given $\alpha=[\alpha_1,\alpha_2]\in\PP^1K$, its \newword{absolute height} is
\[
H(\alpha)=\prod_P\max\bigl(\abs{\alpha_1}_P,\abs{\alpha_2}_P\bigr).
\]
The number $H(\alpha)$ is in $\Rbb_{\ge1}$ and does not depend on $K$, nor on the representative chosen for $\alpha$~\cite[Lemmas~1.5.2 and 1.5.3]{bombierigubler06}.
We let $\Obb$ be the ring of algebraic integers in $K$.

\begin{lemma}
Let\label{ref16} $(\alpha_1,\alpha_2)\in K^2\setminus\set{(0,0)}$, let $A\in\PSL_2^\pm\Obb$, and let $P$ be a nonarchimedean prime. Let $\cvect{\beta_1}{\beta_2}=A\cvect{\alpha_1}{\alpha_2}$. Then $\max\bigl(\abs{\alpha_1}_P,\abs{\alpha_2}_P\bigr)=\max\bigl(\abs{\beta_1}_P,\abs{\beta_2}_P\bigr)$.
\end{lemma}
\begin{proof}
The fractional ideals $\Obb\alpha_1+\Obb\alpha_2$ and $\Obb\beta_1+\Obb\beta_2$ coincide. Since $P$ is nonarchimedean, both $\max\bigl(\abs{\alpha_1}_P,\abs{\alpha_2}_P\bigr)$ and $\max\bigl(\abs{\beta_1}_P,\abs{\beta_2}_P\bigr)$ are equal to the maximum of $\abs{\gamma}_P$, for $\gamma$ varying in that fractional ideal.
\end{proof}

We are interested in the behavior of the height along $T$-orbits of points in $K$, where $T,K$ are as in Definition~\ref{ref8}. Let $A\in\PSL^\pm_2\Obb$; if, up to multiplication by~$-1$, all entries of $A$ are greater than or equal to $0$, with at most one entry equal to $0$, then we call $A$ \newword{positive}. If all entries are greater than $0$ we call $A$ \newword{strictly positive}. Since $r\ge2$, the matrices $A_1,A_r$ are positive (but not strictly), while $A_2,\ldots,A_{r-1}$ are strictly positive. When writing positive matrices, we will always tacitly assume that the displayed entries are greater than or equal to $0$.

The case $K=\Qbb$ is trivial but instructive. The height of $\beta=u/v\in\PP^1\Qbb$ is $\abs{u}\lor\abs{v}$ ($\lor$ denoting the $\max$ function), provided that $u,v$ are relatively prime integers. Assume that $0<\beta<\infty$, and let $A$ be a positive matrix in $\PSL_2^\pm\Zbb$; setting $\alpha=A*\beta$, it is then clear that $H(\alpha)>H(\beta)$.
Since the pieces of $T$ are given by the inverse of positive matrices, 
this implies that, for every $\alpha\in\PP^1\Qbb\cap[0,\infty]$, the height of $T(\alpha)$ is strictly less than the height of $\alpha$, unless $T(\alpha)\in\set{0,\infty}$. As the height is a positive integer, it cannot decrease indefinitely; hence the $T$-orbit of every rational point must end up in $\set{0,\infty}$.

By Northcott's theorem~\cite[Theorem~1.6.8]{bombierigubler06}, the number of points of bounded height in any $\PP^1K$ is finite; the argument above would then work in any number field, if only the action of positive matrices in $\PSL_2^\pm\Obb$ were height-increasing. However, this is not the case.

\begin{example}
Let\label{ref14} $K=\Qbb(\sqrt{6})$, and consider the strictly positive matrices
\[
A_9=
\begin{bmatrix}
2 & -1+\sqrt{6}\\
3+\sqrt{6} & 2+\sqrt{6}
\end{bmatrix},\qquad
A_{13}=\begin{bmatrix}
1+\sqrt{6} & 3-\sqrt{6}\\
2+\sqrt{6} & 1
\end{bmatrix},
\]
introduced in~\S\ref{ref18} when treating Case $(4,6,\infty)$.
The matrix $A_9$ is height-increasing: for every $\beta\in K\cap[0,\infty]$, we have $H(A_9*\beta)>H(\beta)$; this will follow from Theorem~\ref{ref13}.

On the other hand, let
\[
\beta=\frac{703-240\sqrt{6}}{380},\qquad
A_{13}*\beta=
\frac{403+83\sqrt{6}}{346+223\sqrt{6}}.
\]
We can check easily that numerator and denominator in the above expression for~$\beta$ are relatively prime in $\Zbb[\sqrt{6}]$, and thus we compute
\[
H(\beta)=\bigl[380(703+240\sqrt{6})\bigr]^{1/2}=
700.3809\cdots>H(A_{13}*\beta)=422.6795\cdots.
\]
The situation is actually worse: as we will see in Theorem~\ref{ref13}, the set of points in $K$ on which $A_{13}$ is height-decreasing is dense in $[0,\infty]$.
\end{example}

We now introduce our main tool. Let $K$ be any real quadratic field with integer ring $\Obb$, denote by $\phantom{a}'$ the nontrivial automorphism of $K$, and let
$A=\bbmatrix{a}{b}{c}{d}\in\PSL_2^\pm\Obb$ be positive.
We define
\begin{align*}
f(x)&=\frac{(ax+b)\lor(cx+d)}{x\lor 1},\\
t&=\bigl(\abs{a'}+\abs{c'}\bigr)\lor\bigl(\abs{b'}+\abs{d'}\bigr)
=\abs{a'+c'}\lor\abs{a'-c'}\lor\abs{b'+d'}\lor\abs{b'-d'},\\
E^\sharp&=\set{x\in[0,\infty]:f(x)>t},\\
E^\natural&=\set{x\in[0,\infty]:f(x)=t},\\
E^\flat&=\set{x\in[0,\infty]:f(x)<t}.
\end{align*}
When we will have more than one matrix around, we will use notations such as $t(A),E^\sharp(A),f_A$ for clarity.

\begin{theorem}
Assuming\label{ref13} the above notation, the following statements hold.
\begin{enumerate}
\item $f$ is nondecreasing on $[0,1]$ and nonincreasing on $[1,\infty]$.
\item $E^\sharp$ is either empty, or an open interval containing~$1$.
\item Either $E^\natural$ contains at most two points, or equals one of the intervals $[0,1]$, $[1,\infty]$.
\item $H(A*\beta)>H(\beta)$ for every $\beta\in E^\sharp\cap K$; the same holds if $A$ is strictly positive and $\beta$ is a boundary point of $E^\sharp$ such that $A*\beta\not=1$.
\item If $E^\natural$ is an interval, then 
$\set{\beta\in E^\natural\cap K:
H(A*\beta)=H(\beta)}$ is dense in $E^\natural$.
\item $\set{\beta\in E^\flat\cap K:
H(A*\beta)<H(\beta)}$ is dense in $E^\flat$.
\end{enumerate}
\end{theorem}

Assuming Theorem~\ref{ref13}, let
us justify our claims in Example~\ref{ref14}. We have
\[
t(A_9)=-1+\sqrt{6}<\min f_{A_9}=f_{A_9}(0)\land f_{A_9}(\infty)=2+\sqrt{6};
\]
hence $E^\sharp(A_9)=[0,\infty]$ and the statement~(4) applies.
On the other hand,
\[
t(A_{13})=4+\sqrt{6}>\max f_{A_{13}}=f_{A_{13}}(1)=3+\sqrt{6};
\]
hence $E^\sharp(A_{13})=E^\natural(A_{13})=\emptyset$ and $E^\flat(A_{13})=[0,\infty]$, so (6) applies; see Figure~\ref{fig4} (right).

\begin{proof}[Proof of Theorem~\ref{ref13}]
Since $a,b,c,d\ge0$, the statements (1) and (2) are clear. The function $f$ is best visualized by precomposing it with the order-preserving homeomorphism $\varphi:[0,2]\to[0,\infty]$ which is the identity on $[0,1]$ and equals $(2-x)\m$ on $[1,2]$. Then $f\circ\varphi$ is piecewise-linear, namely
\[
(f\circ\varphi)(x)=\begin{cases}
(ax+b)\lor(cx+d),&\text{on $[0,1]$};\\
(-bx+a+2b)\lor(-dx+c+2d),&\text{on $[1,2]$}.
\end{cases}
\]
We have $f(0)=b\lor d$, $f(\infty)=a\lor c$, and $f(1)=(a+b)\lor(c+d)=\max f$. Figure~\ref{fig4} (left) shows $f\circ\varphi$ and $t$ (dashed) for the matrix
\[
A_2A_{10}^9=\begin{bmatrix}
1 & 10+9\sqrt{3}\\
\sqrt{3} & 28+10\sqrt{3}
\end{bmatrix},
\]
where $A_2,A_{10}$ are the matrices introduced in~\S\ref{ref17} when treating Case $(3,6,\infty)$. The dashed vertical segments mark the intervals $I_1,\ldots,I_{10}$ of that case. Figure~\ref{fig4} (right) 
is the analogous one for the matrix $A_{13}$ of Example~\ref{ref14}.
\begin{figure}[h!]
\includegraphics[width=5.5cm]{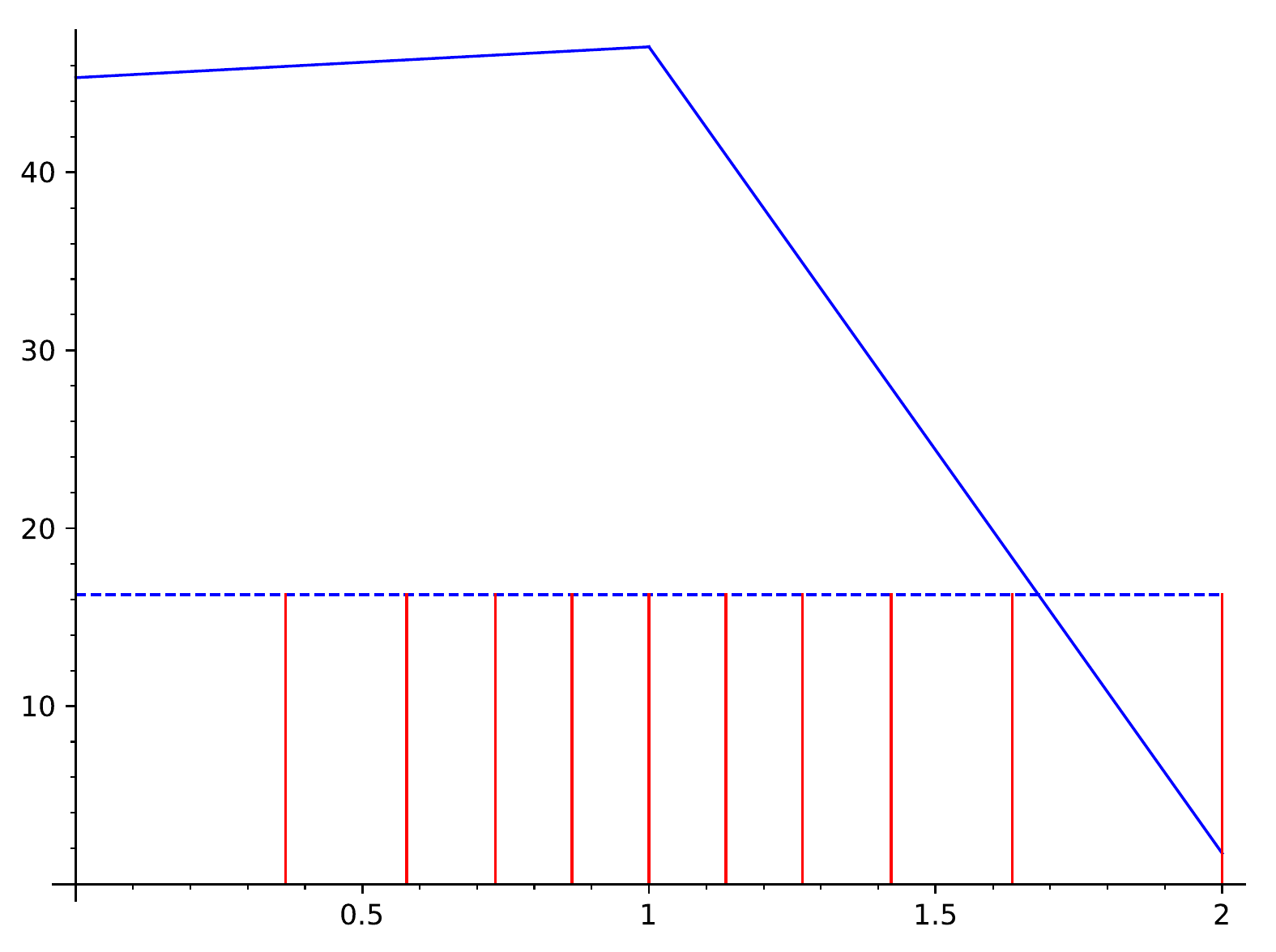}
\hspace{0.6cm}
\includegraphics[width=5.5cm]{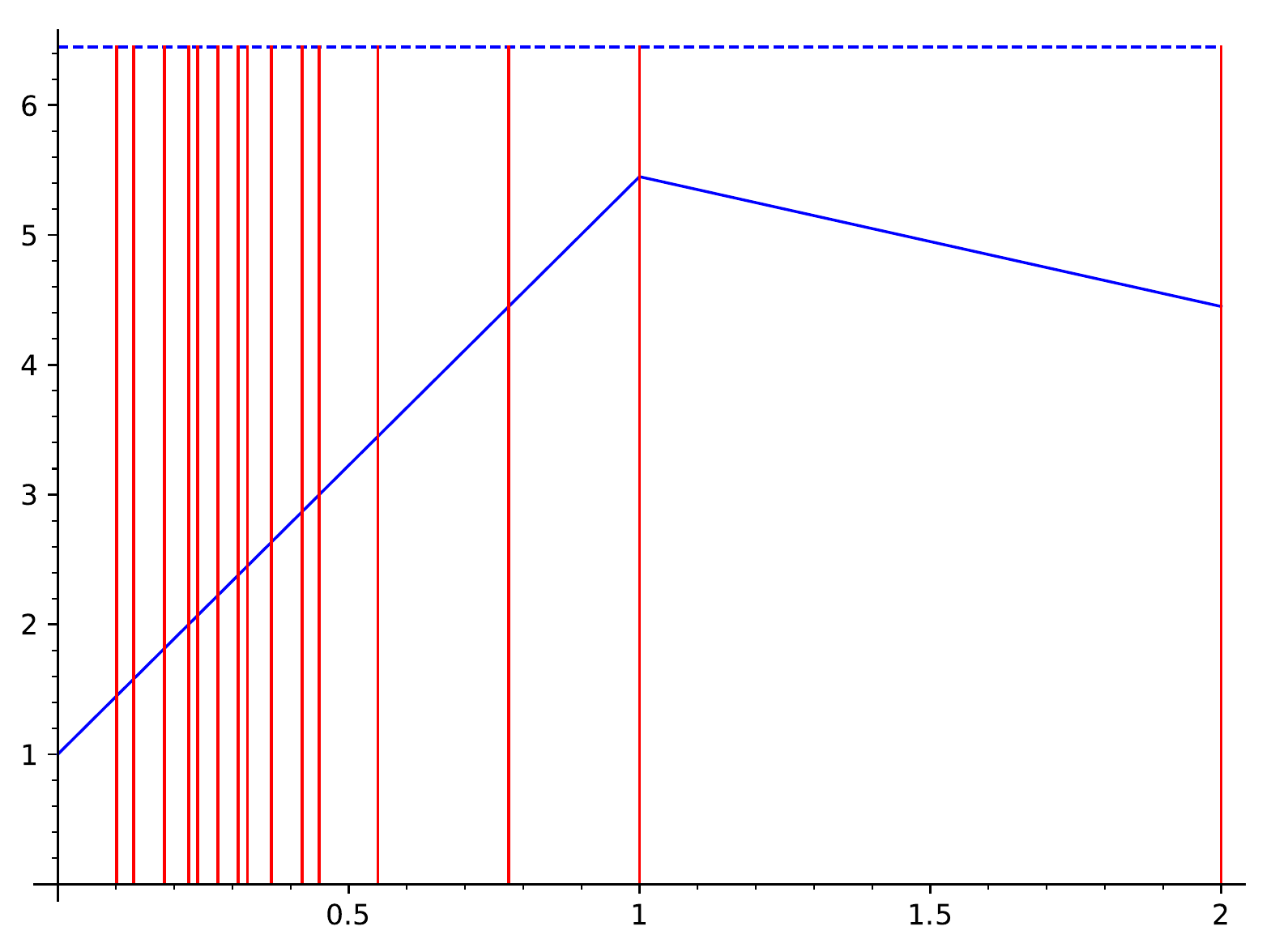}
\caption{Superimposed graphs of $f_A\circ\varphi$ and $t(A)$.\\
Left: $A=A_2A_{10}^9$ of
\S\ref{ref17}; right: $A=A_{13}$ of Example~\ref{ref14}.}
\label{fig4}
\end{figure}

We formulate two remarks.
\begin{itemize}
\item Exchanging the rows of $A$ does not modify $f$ and $t$, and corresponds to postcomposing $A$ with the involution $x\leftrightarrow1/x$, which leaves the height invariant. On the other hand, exchanging the columns does not modify~$t$, and transforms $f(x)$ in $f(1/x)$ (from the point of view of $f\circ\varphi$, it is a reflection in the $x=1$ axis). It corresponds to precomposing $A$ with $x\leftrightarrow1/x$; this simply exchanges $E^\sharp,E^\natural,E^\flat$ with their images under the involution, unaffecting our claims. 
\item The identities $\abs{a'}=\abs{c'}$ and $\abs{b'}=\abs{d'}$
cannot hold simultaneously. This is clear if one of the entries of $A$ is $0$. Otherwise, assume by contradiction that $A$ is strictly positive and both identities hold. Then we have
\begin{align*}
1&=a'd'-b'c'\\
&=\sgn(a'd')\abs{a'}\abs{d'}-\sgn(b'c')\abs{b'}\abs{c'}\\
&=\bigl(\sgn(a'd')-\sgn(b'c')\bigr)\abs{a'}\abs{b'},
\end{align*}
which is impossible, since $\sgn(a'd')-\sgn(b'c')\in\set{-2,0,2}$ and 
$a'b'\not=\pm 1/2$ (because $a',b'$ are algebraic integers).
\end{itemize}
We therefore assume, without loss of generality, that $\abs{a'}<\abs{c'}$ and that $a$ is the only entry of $A$ which can possibly be~$0$.

We prove~(3). If $a\not=0$, then $f$ is strictly increasing on $[0,1]$ and strictly decreasing on $[1,\infty]$; hence $E^\natural$ contains at most $2$ points. 
Otherwise, $A$ has the form $\bbmatrix{}{b}{b\m}{d}$, for some unit~$b\in\Obb$.
Since $E^\natural$ contains $3$ or more points, it contains an interval, that has necessarily the form $[0,r]$, for some $0<r\le1$; indeed, such an interval is the only region where $f$ has derivative~$0$. Looking for a contradiction, we assume $r<1$; then
\[
d<b\m r+d=f(r)=b=t=\abs{(b\m)'}\lor\bigl(\abs{b'}+\abs{d'}\bigr)<f(1)=b\lor(b\m+d).
\]
As $\abs{b'}=b\m$, we obtain
\[
d<b=b\lor\bigl(b\m+\abs{d'}\bigr)<b\lor(b\m+d),
\]
from which we orderly derive
\begin{gather*}
d<b<b\m+d,\\
b\m<b\m+\abs{d'}\le b,\\
b-b\m<d<b,\\
-b+b\m\le d'\le b-b\m,\\
-b\m<d-b<0,\\
-b+b\m-b'\le d'-b'\le b-b\m-b'.
\end{gather*}
Now, $b\m$ equals either $b'$ or $-b'$; in both cases the last line in the above chain yields $\abs{d'-b'}\le b$. Therefore,
writing $N$ for the norm function,
$\bigl\vert N(d-b)\bigr\vert=\abs{d-b}\cdot\abs{d'-b'}<b\m b=1$; as $d-b\in\Obb\setminus\set{0}$ and the norm takes integer values, this is impossible and concludes the proof of~(3).

Let now $\beta$ range in $K\cap[0,\infty)$.
We have $A\cvect{\beta}{1}=\cvect{a\beta+b}{c\beta+d}$; by Lemma~\ref{ref16},
\[
H(A*\beta)\gel H(\beta)
\]
if and only if
\[
\bigl(\abs{a\beta+b}\lor\abs{c\beta+d}\bigr)
\bigl(\abs{a'\beta'+b'}\lor\abs{c'\beta'+d'}\bigr)\gel
\bigl(\abs{\beta}\lor1\bigr)\bigl(\abs{\beta'}\lor1\bigr)
\]
if and only if (since $\beta,a,b,c,d\ge0$)
\begin{equation}\label{eq3}\tag{$*$}
f(\beta)\bigl(\abs{a'\beta'+b'}\lor\abs{c'\beta'+d'}\bigr)\gel\abs{\beta'}\lor 1.
\end{equation}
We note that the corresponding condition for $\beta=\infty$ is
\begin{equation}\label{eq4}\tag{$**$}
f(\infty)\bigl(\abs{a'}\lor\abs{c'}\bigr)\gel 1.
\end{equation}

For a given $x\in[0,\infty]$, 
we denote by $g_x(y)$ the function of the real variable $y$ defined by
\[
g_x(y)=f(x)\bigl(\abs{a'y+b'}\lor\abs{c'y+d'}\bigr).
\]
Since $\abs{a'}<\abs{c'}$, the
graph of $g_x$ has a truncated V shape, with two outer halflines whose slope has absolute value $f(x)\abs{c'}$, and an inner segment whose slope has absolute value $f(x)\abs{a'}$.
These affine linear parts are connected by two distinct \newword{cornerpoints},
which we compute by solving the systems
\[
\begin{cases}
f(x)a'y+f(x)b'=z\\
f(x)c'y+f(x)d'=z
\end{cases}
\quad
\begin{cases}
-f(x)a'y-f(x)b'=z\\
f(x)c'y+f(x)d'=z
\end{cases}
\]
in the variables $(y,z)$ ---say, via the Cramer rule--- and then taking the absolute value of the resulting $z$. After the computation, the cornerpoints turn out to be
\[
q_1=\biggl(\frac{b'-d'}{-a'+c'},\frac{f(x)}{\abs{-a'+c'}}\biggr),
\quad
q_2=\biggl(\frac{-b'-d'}{a'+c'},\frac{f(x)}{\abs{a'+c'}}\biggr).
\]

\paragraph{\emph{Claim}} Let $\beta$ belong to the closure of $E^\sharp$, and let $q_1,q_2$ be the cornerpoints of $g_\beta$; then the following hold.
\begin{itemize}
\item[(i)] The absolute value $f(\beta)\abs{c'}$
of the slopes of the outer halflines of $g_\beta$ is greater than $1$.
\item[(ii)] $q_1$ and $q_2$ belong to the \newword{epigraph}
of $\abs{y}\lor1$, namely the convex region
$\set{(y,z):y+z\ge0,z\ge1,-y+z\ge0}$. If $\beta\in E^\sharp$, then they belong to the \newword{strict epigraph} (defined analogously, with $>$ replacing $\ge$) and hence, by~(i), the entire graph of $g_\beta$ lies in the strict epigraph of $\abs{y}\lor1$.
\item[(iii)] If $a\not=0$ then $q_1,q_2$ do not lie in the same affine linear part of the graph of $\abs{y}\lor1$.
\end{itemize}
\paragraph{\emph{Proof of Claim}} (i) From the triangle inequality we get
\begin{align*}
1=a'd'-b'c'&\le\abs{a'd'}+\abs{b'c'}\\
&<\abs{c'}\bigl(\abs{d'}+\abs{b'}\bigr)\\
&\le\abs{c'}t\\
&\le\abs{c'}f(\beta).
\end{align*}

\noindent (ii) The cornerpoint $q_1$ belongs to the epigraph of $\abs{y}\lor1$
if and only if $\pm(b'-d')+f(\beta)\ge0$ and $f(\beta)\ge\abs{-a'+c'}$,
and this condition amounts to $f(\beta)\ge\abs{a'-c'}\lor\abs{b'-d'}$. Analogously, 
$q_2$ belongs to the epigraph if and only if $f(\beta)\ge\abs{a'+c'}\lor\abs{b'+d'}$.
By our assumptions on $\beta$ we have $f(\beta)\ge t$, and the above inequalities hold by definition of $t$. If $\beta\in E^\sharp$ then $f(\beta)<t$ and the inequalities are strict.

\noindent (iii) Suppose that $q_1,q_2$ lie on $\set{-y+z=0}$. Then $\abs{b'-d'}=f(\beta)=\abs{b'+d'}$, which is impossible since $b,d$ are both different from $0$. The same argument works for the linear parts along $\set{z=1}$ and $\set{y+z=0}$. \qed

We now prove (4). For $\beta=\infty$, the greater-than inequality in~\eqref{eq4} is just~(i) in the Claim; let then $\beta\not=\infty$.
If $\beta\in E^\sharp$ then, by (ii) in the Claim,
$g_\beta(y)>\abs{y}\lor1$ for every $y\in\Rbb$. In particular, $g_\beta(\beta')>\abs{\beta'}\lor1$, which amounts to the greater-than inequality in~\eqref{eq3}. Assume that $A$ is strictly positive and that $\beta$ is a boundary point of $E^\sharp$ with $A*\beta\not=1$; then $f(\beta)=t$. By (i) and (iii) of the Claim, the graph of $g_\beta$ is contained in the strict epigraph of $\abs{y}\lor1$, with the exception of at least one and at most two of $q_1,q_2$. The first coordinates of these cornerpoints are $(A\m*1)'$ and $\bigl(A\m*(-1)\bigr)'$. Neither of these can be $\beta'$, since otherwise we would have $A*\beta=1$ ---which is excluded by hypothesis--- or $A*\beta=-1$ ---which is impossible---. Therefore $g_\beta(\beta')>\abs{\beta'}\lor1$ holds as well and (4) is established.

We prove (5); let $J\subseteq E^\natural=[0,1]$ be an interval of positive length. By the proof of (3), $f$ has value $b$ on $J$. Thus, for every $x\in J$, the cornerpoints of $g_x$ are
\[
q_1=\biggl(\frac{b'-d'}{N(b)b},1\biggr),
\quad
q_2=\biggl(\frac{-b'-d'}{N(b)b},1\biggr).
\]
Indeed, under our hypotheses, $a$ equals $0$ and $b$ is a unit, so that $c'=(b\m)'=(b\m)'bb\m=N(b\m)b=N(b)b$.
Since $\abs{b'-d'}\lor\abs{-b'-d'}\le t=b$, the interval $J'$ of endpoints $\frac{b'-d'}{N(b)b}$ and $\frac{-b'-d'}{N(b)b}$ is contained in $[-1,1]$ and has positive length. By the Artin-Whaples approximation theorem~\cite[Chapter~XII, Theorem~1.2]{lang},
there exists $\beta\in K$ such that $\beta\in J$ and $\beta'\in J'$. Therefore $g_\beta(\beta')=1=\abs{\beta'}\lor1$, which means equality in~\eqref{eq3} and establishes (5).

Finally, let $J\subseteq E^\flat$ be an interval of positive length; by possibly shrinking it, we find $m<t$ such that $f\le m$ on $J$. By definition of $t$, either $m<\abs{a'+c'}\lor\abs{a'-c'}$ or $m<\abs{b'+d'}\lor\abs{b'-d'}$ (or both). In any case, at least one of the cornerpoints of the graph of $m\bigl(\abs{a'y+b'}\lor\abs{c'y+d'}\bigr)$ is strictly below the graph of $\abs{y}\lor1$; in particular, there exists an interval of positive length $J'$ such that $g_x(y)<\abs{y}\lor1$ for every pair $(x,y)\in J\times J'$. Again by the Artin-Whaples theorem, there exists $\beta\in K$ with $(\beta,\beta')\in J\times J'$; hence $g_\beta(\beta')<\abs{\beta'}\lor1$ and $H(A*\beta)<H(\beta)$. This establishes (6) and concludes the proof of Theorem~\ref{ref13}.
\end{proof}

\section{Decreasing blocks}\label{ref5}

Let $T$ and $K$ be a Gauss map and a number field as in Definition~\ref{ref8}. For a ``generic'' input $\alpha\in K$, the height along the $T$-orbit of $\alpha$ is ``generically decreasing''. This is of course a vague statement, but is exemplified in Figure~\ref{fig5} (left), which shows the logarithmic height (i.e., the logarithm of the height) along the first $180$ points in the $T$-orbit of~$\alpha$, where
\begin{itemize}
\item
$T$ is the map introduced in~\S\ref{ref17} for the $(3,4,\infty)$ case; \item $\alpha$ is the rational number
\[
\alpha=\frac{17872569142100116747518989441504411}{5689015449433817447221222422292025},
\]
whose only ---as far as we know--- claim to glory is being the $66$th convergent of the ordinary c.f.~expansion of~$\pi$.
\end{itemize}
In this final section we will give a precise formulation to property~(H), and prove it for all the algorithms in \S\ref{ref17}, \S\ref{ref6}, \S\ref{ref18}. This will remove, in a quite effective way, all vagueness from the above assertion. Note that, as remarked in~\S\ref{ref1}, the validity of~(H) implies that any point in $K$ must end its orbit in a parabolic fixed point; explicit computation shows that indeed $\alpha$ lands at~$\infty$ at the $656$th iterate of~$T$.
\begin{figure}[h!]
\includegraphics[width=5.8cm]{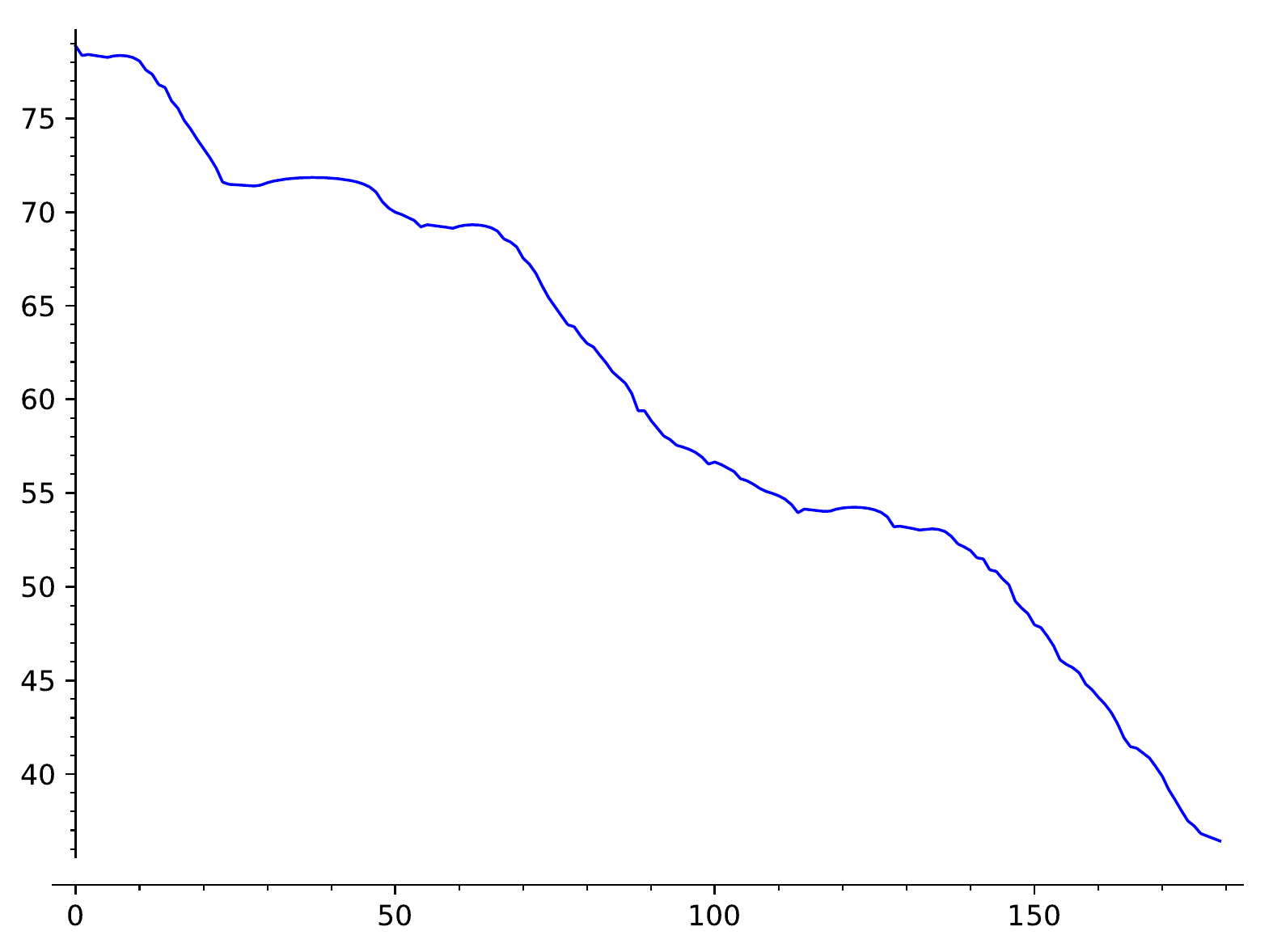}
\quad
\includegraphics[width=5.8cm]{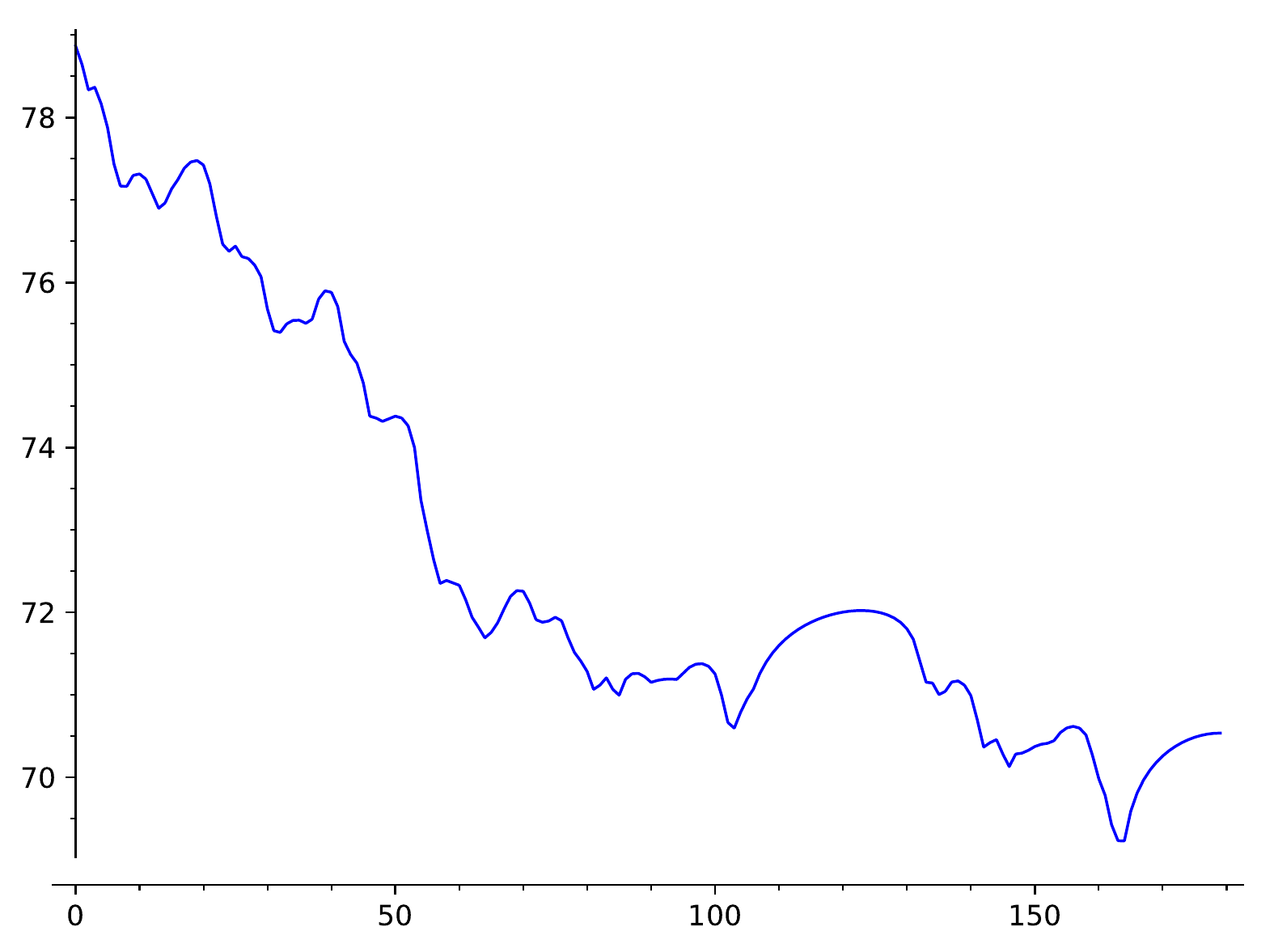}
\caption{Logarithmic heights along an orbit.\\
A quadratic (left) vs.~a cubic case (right).}
\label{fig5}
\end{figure}

It is interesting to compare the behavior of the previous quadratic $T$ with that of a cubic map. 
Let $\lambda_7=2\cos(\pi/7)$, whose minimal polynomial is
$x^3-x^2-2x+1$; again the maximal order $\Zbb[\lambda_7]$ in $\Qbb(\lambda_7)$ has class number~$1$.
Precisely as in~\S\ref{ref17} for the case~$(2,5,\infty)$, we define a Gauss map $T'$ via the matrices
\[
R_7Q_2F,\;\;R_7^2Q_2,\;\;R_7^3Q_2F,\;\;R_7^4Q^2,\;\;R_7^5Q_2F,\;\;R_7^6Q_2,
\]
which generate a copy of the extended Hecke group $\Delta^\pm(2,7,\infty)$. The graphs of $T$ and $T'$ have the same structure; they are both saw-like, with $6$ pieces and $\infty$ as the only parabolic fixed point (similar to Figure~\ref{ref1} (right), but with the leftmost piece deleted).
Figure~\ref{fig5} (right) shows the behavior of the logarithmic height along the first $180$ points of the $T'$-orbit of the same input $\alpha$. The decreasing drift is clearly much more erratic, and indeed $\alpha$ takes $4583$ steps to reach $\infty$. Note that, as far as we know, there is no a priori guarantee that $\alpha$ should land at $\infty$, nor that it is a cusp of $\Delta(2,7,\infty)$. Indeed, property (H) definitely fails: the point $\lambda_7^2-\lambda_7-1$ (which has minimal height greater than $1$ in $K=\Qbb(\lambda_7)$) is a fixed point for $T'$, because it is in $I_1$ and is fixed by $A_1=R_7Q_2F=\bbmatrix{}{1}{1}{\lambda_7}$. As it is fixed by the hyperbolic $A_1^2$, it is not a cusp of $\Delta(2,7,\infty)$.
This is an instance of a fixed point for a so-called special hyperbolic matrix, which corresponds to an affine pseudo-Anosov diffeomorphism with direction of vanishing SAF invariant;
see~\cite[\S5]{arnouxschmidt09} and references therein.

Given a Gauss map $T$ as in Definition~\ref{ref8}, we regard $\Afrak=\set{1,\ldots,r}$ as a finite alphabet; as usual, $\Afrak^*$ denotes the set of all finite words over $\Afrak$, while $\Afrak^\omega$ is the space of all one-sided infinite sequences.
For every $\alpha\in[0,\infty]$, a \newword{$T$-symbolic orbit} for $\alpha$ is a sequence $\abf=a_0a_1\cdots\in\Afrak^\omega$ such that $T^k(\alpha)\in I_{a_k}$, for $k\ge0$.
Every point has at least one and at most two symbolic orbits; it has two if and only if it eventually lands at an endpoint of one of $I_2,\ldots,I_{r-1}$ (so that, at the next step, gets mapped either to $0$ or to $\infty$).
Note that this ambiguity is in the nature of things; if the map $\abf\mapsto\alpha$ were injective then, since it it surjective and obviously continuous, it would constitute a homeomorphism from the Cantor space $\Afrak^\omega$ to $[0,\infty]$. 
The symbolic orbit $\overline{r}$ of $\infty$ is however unique, and so is that of $0$, namely $\overline{1}$ or $1\overline{r}$ according whether $\det A_1$ equals $+1$ or $-1$.

The following result generalizes~\cite[Lemma~2.2.15]{smillieulcigrai11},
as well as~\cite[Observation~3]{panti09}, to arbitrary extended fuchsian groups; note that the proof uses a different technique.

\begin{theorem}
For\label{ref15} every $\abf\in\Afrak^\omega$, the nested strictly decreasing sequence of $\Gamma^\pm$-unimodular intervals
\[
I_0\supset A_0*I_0\supset A_0A_1*I_0\supset A_0A_1A_2*I_0\supset\cdots
\]
shrinks to a singleton $\set{\alpha}$. The point $\alpha$ is the unique point having symbolic orbit~$\abf$.
\end{theorem}
\begin{proof}
The second statement is an obvious consequence of the first. The intersection of the above sequence of intervals is surely a closed interval $[\alpha,\beta]$, with $0\le\alpha\le\beta\le\infty$; we show that $\alpha=\beta$. Suppose not, and let $B_k=A_0\cdots A_k$. Then there exists a strictly increasing sequence $k_0,k_1,\ldots$ such that $B_{k_l}*0$ converges, as $l\to\infty$, to one of $\alpha,\beta$, say to $\alpha$, and $B_{k_l}*\infty$ converges to $\beta$. Since the projective space $\PP\Mat_2\Rbb=\PP^3\Rbb$ is compact, we can find a subsequence $(k_{l_j})_{j\in\omega}$ such that $\lim_{j\to\infty}[B_{k_{l_j}}]=[C]$, for some $C\in(\Mat_2\Rbb)\setminus{0}$, square brackets denoting projective equivalence.
As $C*0=\alpha\not=\beta=C*\infty$, the matrix $C$ is nonsingular and determines an isometry of $\Hcal$, possibly orientation-reversing. Now, the point $C*i$ belongs to $\Hcal$, and is an accumulation point of the sequence $(B_{k_{l_j}}*i)_{j\in\omega}$. Since the $B_k$'s belong to $\Gamma^\pm$ and are all distinct, this contradicts the fact that $\Gamma^\pm$ acts on $\Hcal$ in a properly discontinuous way.
\end{proof}

\begin{definition}
Let\label{ref19} $T$ be a continuous quadratic Gauss map as in Definition~\ref{ref8}. A \newword{decreasing block} is a word $b_0\cdots b_u\in\Afrak^*$, of length $u+1\ge2$, such that:
\begin{itemize}
\item[(i)] $b_0,b_u\not=r$;
\item[(ii)] $I_{b_u}$ is contained in the closure of $E^\sharp(A)$, where $A=A_{b_0}\cdots A_{b_{u-1}}=\bbmatrix{a}{b}{c}{d}$;
\item[(iii)] if $A$ contains a $0$ entry (this happens if and only if $\det A_1=+1$ and $A$ is a power of $A_1$, or $\det A_1=-1$ and $A=A_1$) and $\beta\in\partial E^\sharp(A)\cap\partial I_{b_u}$, then
\[
t(A)\bigl(\abs{a'\beta'+b'}\lor\abs{c'\beta'+d'}\bigr)>\abs{\beta'}\lor1.
\]
\end{itemize}
A set $\Bcal$ of decreasing blocks is \newword{complete} for $T$ if:
\begin{itemize}
\item[(iv)] $\det A_1=-1$ and every $\abf\in\Afrak^\omega$ that does not start with $r$, and does not end with $\overline{r}$, begins with a block in $\Bcal$;
\item[(v)] $\det A_1=+1$ and every $\abf\in\Afrak^\omega$ that does not start with $r$, and does not end either with $\overline{r}$ or with $\overline{1}$, begins with a block in $\Bcal$.
\end{itemize}
\end{definition}

We can at last formulate property~(H) in a precise way.

\begin{theorem}
Each\label{ref20} of the algorithms introduced in \S\ref{ref17}, \S\ref{ref6}, \S\ref{ref18} admits a complete set of decreasing blocks.
\end{theorem}

The content of Theorem~\ref{ref20} is the following: choose any symbolic sequence $\abf\in\Afrak^\omega$ that does not end in $\overline{r}$ (nor in $\overline{1}$, if $\det A_1=+1$). Then, by Theorem~\ref{ref20}, there exists an infinite sequence of times $0\le t_0<t_1<t_2\cdots$ such that each block $a_{t_k}\cdots a_{t_{k+1}}$ is decreasing. Choose any initial very long finite subword $w$ of $\abf$, and let $\alpha$ be any element of $K$ having a $T$-symbolic orbit that starts with $w$; let $h=\max\set{l:t_l<\text{length}(w)}$.
Then the following fact is true.

\paragraph{\emph{Claim}} For every $0\le k<h$, the strict inequality
\[
H\bigl(T^{t_k}(\alpha)\bigr)>
H\bigl(T^{t_{k+1}}(\alpha)\bigr),
\]
holds, with the only possible exception of the last index $k=h-1$, where we may have equality; however, this exception takes place if and only if $T^{t_{h-1}}(\alpha)=1$, and forces $T^{t_h}(\alpha)=0$.
\paragraph{\emph{Proof of Claim}} 
Let us fix a decreasing block $b_0\cdots b_u$ 
and $\gamma\in K$ having a symbolic orbit that starts with 
$b_0\cdots b_u$. Let $A=A_{b_0}\cdots A_{b_{u-1}}$ and 
$\beta=T^u(\gamma)$; then $\beta\in I_{b_u}$ and $\gamma=A*\beta$.
By the definition of decreasing block, $\beta$ is in the closure of $E^\sharp(A)$. 
If $\beta\in E^\sharp(A)$, or ($\beta\in\partial E^\sharp(A)$, $A$ is strictly positive, and $A*\beta\not=1$), then $H(A*\beta)>H(\beta)$ by Theorem~\ref{ref13}(4). If $\beta\in\partial E^\sharp(A)$ then $f_A(\beta)=t(A)$; also, by Definition~\ref{ref19}(ii), $\beta\in\partial I_{b_u}$. If $A$ contains a $0$ entry then, by Definition~\ref{ref19}(iii),
\[
f_A(\beta)\bigl(\abs{a'\beta'+b'}\lor\abs{c'\beta'+d'}\bigr)>\abs{\beta'}\lor1,
\]
which is~\eqref{eq3} in the proof of Theorem~\ref{ref13}; thus,
$H(A*\beta)>H(\beta)$ again. In the exceptional case in which $\beta\in\partial E^\sharp(A)$, $A$ is strictly positive, and $A*\beta=1$, we have $H(\gamma)\ge H(\beta)$ anyway, so $H(A*\beta)=H(\beta)=1$ and $\beta\in\set{0,1,\infty}$. Since $b_u\not=r$, the case $\beta=\infty$ is excluded. If $\beta$ were equal to~$1$, then $1$ would be fixed 
by a strictly positive matrix in $\Gamma$, namely $A$ or $A^2$.
It is easy to show that this contradicts our assumption that $1$ is a cusp of $\Gamma$; we thus conclude $\beta=0$. Note that the exceptional case may occur only at the final decreasing block in $w$, because as soon as $T^t(\alpha)$ reaches value $0$, it can never assume value $1$ anymore. Our claim is thus established. \qed

We now state formally the geometric consequence of property~(H) sketched in~\S\ref{ref1}.
Let $X$ be a subset of $\PP^1\Rbb$, of cardinality $\sharp X$ at least~$3$. We let $\CR(X)$ be the set of all cross-ratios $(\alpha_1,\alpha_2;\alpha_3,\alpha_4)$, for $\alpha_i\in X$ and $\sharp\set{\alpha_1,\alpha_2,\alpha_3,\alpha_4}\ge3$;
we have $\CR(X)\subseteq\PP^1\Rbb$, and $(\alpha_1,\alpha_2;\alpha_3,\alpha_4)\in\set{0,1,\infty}$ if and only if $\sharp\set{\alpha_1,\alpha_2,\alpha_3,\alpha_4}=3$.

\begin{theorem}
Let~\label{ref22} $\Gamma$ be a 
noncocompact triangle group with quadratic invariant trace field~$F_2$.
Then $\CR(\text{cusps of $\Gamma$})=\PP^1F_2$.
\end{theorem}
\begin{proof}
The set $\CR(\text{cusps of $\Gamma$})$ is a commensurability invariant of~$\Gamma$; indeed, it does not change either by conjugating~$\Gamma$ (by the invariance of cross-ratios under projective transformations), or by passing to subgroups of finite index (by the remark at the beginning of~\S\ref{ref2}). We can then assume that~$\Gamma$ equals $\Gamma_T$, for precisely one of our Gauss maps~$T$.
By construction, all matrices in~$\Gamma_T$ have entries in the integer ring of~$F_2$, and this easily implies that every cusp of~$\Gamma_T$ is in~$\PP^1F_2$. For the reverse inclusion, let $\alpha\in\PP^1F_2$. By applying a power of~$A_r$ (the parabolic matrix in Definition~\ref{ref8}(v)) to~$\alpha$, we may assume~$\alpha\in[0,\infty]$. By Theorem~\ref{ref20} and the Northcott property, there exists a product~$A$ of $A_1\m,\ldots,A_r\m\in\Gamma_T^\pm$ such that $A*\alpha$ equals~$0$ or~$\infty$. Postcomposing~$A$, if necessary, with an element of $\Gamma_T^\pm$ that fixes $A*\alpha$ and has negative determinant, we may assume $A\in\Gamma_T$. As $0$ and $\infty$ are cusps of $\Gamma_T$, so is $\alpha$.
\end{proof}

\begin{remark}
As noted in~\S\ref{ref1}, Theorem~\ref{ref22} was proved by Leutbecher for the four quadratic Hecke groups, and reestablished by McMullen~\cite[Theorem~A.2]{mcmullen03}, using the Veech theory, for three more groups, namely~$\Delta(3,4,\infty)$, $\Delta(3,5,\infty)$, $\Delta(3,6,\infty)$; these three groups had been previously examined by the Leutbecher school~\cite{berg85}, \cite{seibold85}.
As pointed out by the referee, \cite[Theorem~A.1]{mcmullen03} and the
further realizations of Veech groups in~\cite{bouwmoller10} yield the full result.
\end{remark}

Theorem~\ref{ref20} is proved by direct verification.
One computes $E^\sharp(A_a)$, for each $a=1,\ldots,r-1$;
whenever its closure contains some $I_b$, with $b\not=r$, the block $ab$ is decreasing (for $a=1$ boundary points must be checked according to Definition~\ref{ref19}(iii)). If $I_b$ is not contained in the closure of $E^\sharp(A_a)$, then one computes $E^\sharp(A_aA_b)$ and repeats the process. After some practice a good candidate for $\Bcal$ is readily reached, and checking that indeed it constitutes a complete set of decreasing blocks is just a patience task. The simplest cases are those in which the intervals are evenly spread around~$1$, namely those in \S\ref{ref17}. The cases in~\S\ref{ref6} are slightly more involved, 
due to the presence of a second parabolic fixed point at $0$ (note that~$\infty$ and $0$ do not play a symmetric r\^ole).
The case $(4,6,\infty)$ is the most difficult one, since all intervals $I_1,\ldots,I_{14}$ cumulate in $[0,1]$, and several $E^\sharp$'s are small. Despite the larger number of intervals, case $(4,12,\infty)$ is actually simpler, because intervals are more evenly spread, and the central ones behave in the same way. We start from the simplest case.

\subsection{Case $(2,5,\infty)$} The matrices are
\[
A_1=\begin{bmatrix}
&1\\
1&\tau
\end{bmatrix},\quad
A_2=\begin{bmatrix}
\tau&1\\
\tau&\tau
\end{bmatrix},\quad
A_3=\begin{bmatrix}
\tau&\tau\\
\tau&1
\end{bmatrix},\quad
A_4=\begin{bmatrix}
1&\tau\\
&1
\end{bmatrix},
\]
and the $4$ intervals have endpoints $0<\tau\m<1<\tau<\infty$.
Let $\Bcal$ be the family of all blocks of the form $a4^kb$, where $a,b\not=4$ and $k\ge0$.
For convenience, we introduce shorthands for block families; notations such as $\set{a,b}^c\,c\set{d,e}f^*\set{f,g}^c$ will denote the set of all blocks of the form $xcyf^kz\in\Afrak^*$ such that $x\in\set{a,b}^c=\Afrak\setminus\set{a,b}$, $y\in\set{d,e}$, $z\in\set{f,g}^c$, and $k\ge0$. The completeness of the above $\Bcal$ (which, in the previous notation, is $\set{4}^c4^*\set{4}^c$)
is clear: every $\abf\in\set{1,2,3,4}^\omega$ that does not start with $4$ and does not end with $\overline{4}$ ($\det A_1=-1$ here) must begin with a block in $\Bcal$.
For $k\ge0$ we have
\[
A_1A_4^k=\begin{bmatrix}
&1\\
1&(k+1)\tau
\end{bmatrix},\quad
A_2A_4^k=\begin{bmatrix}
\tau&(k+1)+k\tau\\
\tau&k+(k+1)\tau
\end{bmatrix}.
\]

We compute $t(A_1A_4^k)=1+(k+1)\tau\m$, $E^\sharp(A_1)=(0,\tau^2)$, and $E^\sharp(A_1A_4^k)=[0,\tau^2)$ for $k\ge1$. We must make sure that the boundary point $0\in\partial E^\sharp(A_1)\cap\partial I_1$
satisfies Definition~\ref{ref19}(iii), which is indeed the case since $t(A_1)\bigl(\abs{a'0+b'}\lor\abs{c'0+d'}\bigr)=(1+\tau\m)(1\lor\abs{\tau'})=1+\tau\m>1$. Since $\tau<\tau^2$, we have $I_1\cup I_2\cup I_3=[0,\tau]\subset[0,\tau^2]=(\text{closure of $E^\sharp(A_1A_4^k)$})$, and all blocks of the form $14^*\set{4}^c$ are decreasing. 

The number $t(A_2A_4^k)$ has value $\tau$ for $k=0,1$, and value $(2k+1)(2-\tau)$ for $k\ge2$. We have $E^\sharp(A_2)=(0,\xi(0))$ and $E^\sharp(A_1A_4^k)=[0,\xi(k))$ for some number $\xi(k)$ which we do not need to compute since it is surely greater than $\tau$; indeed, $f_{A_2A_4^k}(\tau)=\tau+k/\tau+(k+1)$, which is strictly greater than $t(A_2A_4^k)$ for every $k$. Again, $0\in\partial E^\sharp(A_2)\cap \partial I_1$, so we must check the inequality in Definition~\ref{ref19}(iii), i.e., that
$t(A_2)\bigl(1\lor\abs{\tau'}\bigr)>1$, which is the case. We conclude that all blocks $24^*(4)^c$ are decreasing. Since $A_3A_4^k$ is obtained from $A_2A_4^k$
by exchanging rows and this operation, as remarked in the proof of Theorem~\ref{ref13}, leaves $t$ and $f$ unaltered, all blocks $34^*\set{4}^c$ are decreasing as well. This concludes the proof of Theorem~\ref{ref20} for the map of case $(2,5,\infty)$.

The particular form of the blocks $\set{4}^c4^*\set{4}^c$ implies that, starting with any input in $[0,\tau]\cap\Qbb(\tau)$ and waiting until the $T$-orbit returns to $[0,\tau]$ guarantees that the height decreases strictly. Thus, the \newword{first return map} induced by $T$ on $[0,\tau]$ (see Figure~\ref{fig6}) is strictly height-decreasing. In particular, 
the first-return orbit of every point in $K=\Qbb(\tau)$ is eventually undefined: we have to wait forever.
\begin{figure}[h!]
\includegraphics[width=6.5cm]{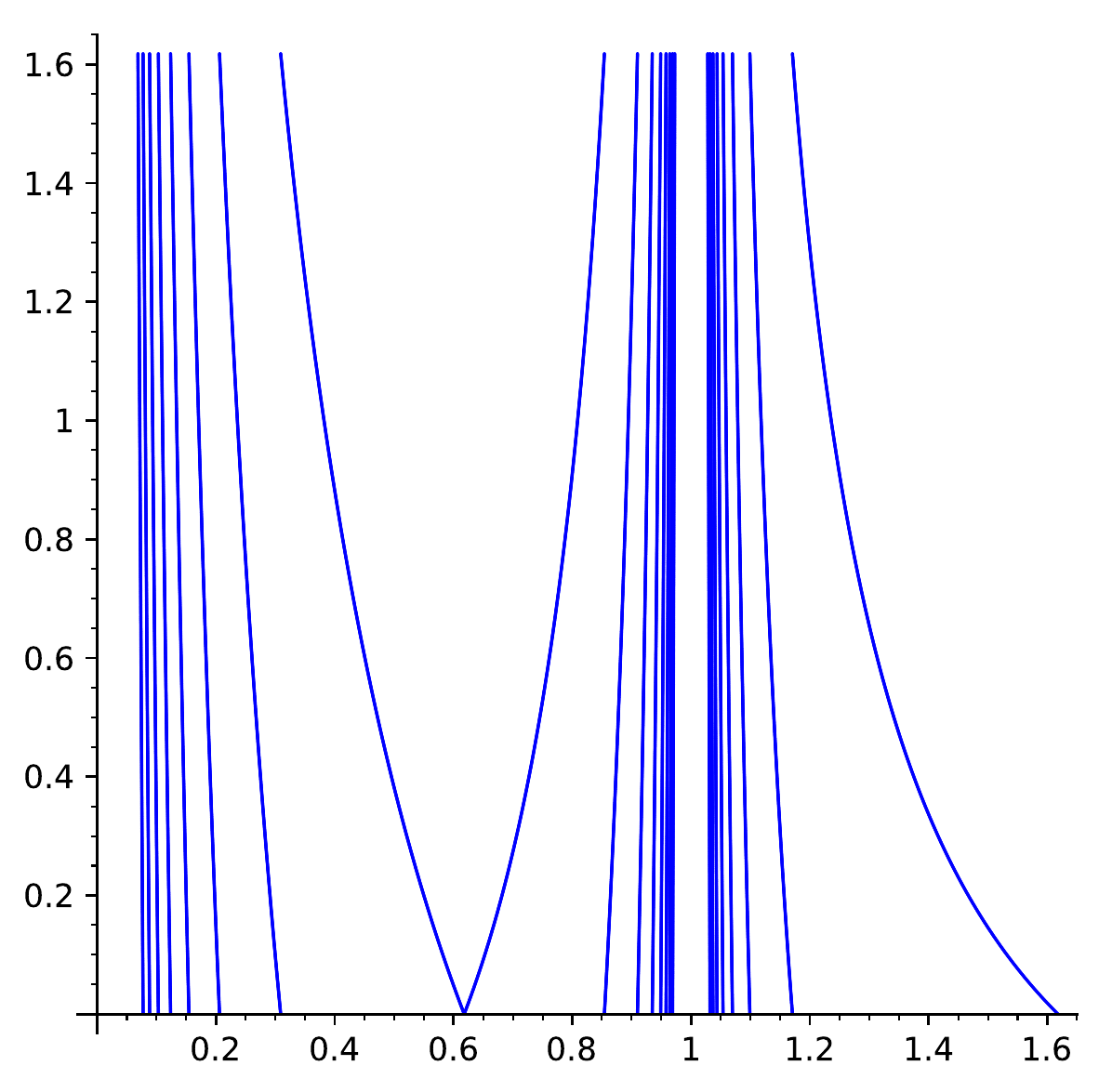}
\caption{The first-return map in the $(2,5,\infty)$ case.}
\label{fig6}
\end{figure}

\begin{remark}
In general, conjugating a map changes radically the structure of decreasing blocks. For example, conjugating the map $T$ above by $x\mapsto \tau\m x$ we obtain a Gauss map $T'$ for which the block $11$ is not decreasing any more.
\end{remark}

\subsection{Cases $(3,4,\infty)$ and $(3,5,\infty)$} 
These cases  are identical to case $(2,5,\infty)$. The set of all blocks of the form $\set{5}^c5^*\set{5}^c$  (respectively, $\set{7}^c7^*\set{7}^c$) is a complete family of decreasing blocks, and the map of first-return on the first $r-1$ intervals is strictly height-decreasing.

\subsection{Case $(3,6,\infty)$} 
This one is more involved. There are $10$ intervals, the rightmost two being $I_9=[\sqrt{3},1+\sqrt{3}]$ and $I_{10}=[1+\sqrt{3},\infty]$. The blocks $1(10)^*\set{10}^c$ are height-decreasing (we make appeal to the reader's patience for the extra parentheses, which are needed to avoid digit ambiguities).
Now, computing $E^\sharp$ for the matrix
\[
A_2A_{10}^k=\begin{bmatrix}
1&(k+1)+k\sqrt{3}\\
\sqrt{3}&(3k+1)+(k+1)\sqrt{3}
\end{bmatrix}
\]
yields an interval $[0,\xi(k))$. The number $\xi(k)$ is in $I_{10}$ for $k=0,\ldots,16$ (see Figure~\ref{fig4} (left)), but enters $I_9$ from $k=17$ on, converging monotonically from the right to $(3+\sqrt{3})/2$, which is in the topological interior of $I_9$. It follows that the blocks $2(10)^k\set{10}^c$ are indeed decreasing for $k\le16$, but no more so for higher $k$'s. We therefore include in $\Bcal$ all blocks of the form $2(10)^k\set{10}^c$ for $k\le16$ and of the form $2(10)^k9(10)^*\set{10}^c$ for $k\ge17$; computing the relative $E^\sharp$'s shows that the latter blocks are decreasing.
Identical considerations apply to blocks starting with the digits $a=3,\ldots,9$. Note that if we are satisfied with a coarser $\Bcal$ (i.e., a set of decreasing blocks which, albeit complete, is nonoptimal, as it contains blocks longer than necessary), then we may avoid computing the minimal $k$ corresponding to each $a$, and just put in $\Bcal$, in addition to $1(10)^*\set{10}^c$, all blocks of the form $\set{1,10}^c(10)^*\set{9,10}^c$ 
and $\set{1,10}^c(10)^*9(10)^*\set{10}^c$.

In case $(3,6,\infty)$, the first-return map to $I_1\cup\cdots\cup I_9=[0,1+\sqrt{3}]$ is \emph{not} height-decreasing. The point
\[
\alpha=\frac{-106298232+90029733\sqrt{3}}{-87526426+100742393\sqrt{3}}
\]
is in $I_2$ and has logarithmic height $18.8328\cdots$. It gets mapped to $I_{10}$ by $T$, and stays there for $16$ parabolic steps, returning to $I_9$ as
\[
\beta=A_2\m*\alpha-17(1+\sqrt{3})=
\frac{-284852703+186193328\sqrt{3}}{-18821875+18821875\sqrt{3}},
\]
which has logarithmic height $18.8341\cdots$.

We have thus proved Theorem~\ref{ref20} for the cases in~\S\ref{ref17}.
In the remaining cases both extremal intervals $I_1$ and $I_r$ correspond to a parabolic matrix; note again that in the construction of $\Bcal$ the r\^ole of the digits $1$ and $r$ is not symmetric.

\subsection{Cases $(4,\infty,\infty)$, $(5,\infty,\infty)$, and $(6,\infty,\infty)$} These are similar, albeit not completely overlapping. The alphabets are $\Afrak=\set{1,\ldots,r}$, with $r$ equal to $7$, $9$, and $11$, respectively. In each case we provide a list of block families whose union gives a complete set of decreasing blocks. The completeness of the sets is apparent, and the fact that the blocks are decreasing reduces to a tedious but straightforward computation of the relative~$E^\sharp$'s.
\begin{itemize}
\item $(4,\infty,\infty)$
\begin{gather*}
11^*\set{1,7}^c,\\
11^*77^*\set{7}^c,\\
\set{1,6,7}^c\,7^*\set{7}^c,\\
6\set{1,7}^c,\\
611^*\set{1,2,7}^c,\\
611^*\set{2,7}7^*\set{7}^c,\\
677^*\set{7}^c.
\end{gather*}
\item $(5,\infty,\infty)$
\begin{gather*}
11^*\set{1,2,9}^c,\\
11^*\set{2,9}9^*\set{9}^c,\\
\set{1,8,9}^c\,9^*\set{9}^c,\\
8\set{1,9}^c,\\
811^*\set{1}^c\,9^*\set{9}^c,\\
899^*\set{9}^c.
\end{gather*}
\item $(6,\infty,\infty)$
\begin{gather*}
11^*\set{1,11}^c,\\
11^*(11)(11)^*\set{11}^c,\\
\set{1,10,11}^c\,(11)^*\set{11}^c,\\
(10)\set{1,11}^c,\\
(10)11^*\set{1,11}^c,\\
(10)11^*(11)(11)^*\set{11}^c,\\
(10)(11)(11)^*\set{11}^c.
\end{gather*}
\end{itemize}

\subsection{Case $(4,12,\infty)$} This case involves the largest alphabet $\Afrak=\set{1,\ldots,33}$ but it is not difficult, since all digits from $2$ to $29$ behave in the same way.
\begin{gather*}
11^*\set{1,33}^c,\\
11^*(33)(33)^*\set{33}^c,\\
\set{1,30,31,32,33}^c\,(33)^*\set{33}^c,\\
(30)\set{1,33}^c,\\
(30)11^*(33)^*\set{33}^c,\\
(30)(33)(33)^*\set{33}^c,\\
\set{31,32}\set{1,2,3,33}^c,\\
\set{31,32}11^*\set{1,33}^c,\\
\set{31,32}11^*(33)(33)^*\set{33}^c,\\
\set{31,32}\set{2,3,33}(33)^*\set{33}^c.
\end{gather*}

\subsection{Case $(4,6,\infty)$} The alphabet is $\Afrak=\set{1,\ldots,15}$. This is the most difficult case, and the unique one in which a matrix, namely $A_{13}$ of Example~\ref{ref14}, has its~$E^\flat$ equal to all of $[0,\infty]$; see Figure~\ref{fig4} (right). The closure of the set of points in $K=\Qbb(\sqrt{6})$ on which $T$ is strictly increasing is also large, and includes all of~$I_{13}$; however, further applications of $T$ make up for these temporary increases. The following is a list of block families whose union provides a complete set of decreasing blocks.
\begin{gather*}
11^*\set{1,15}^c,\\
11^*(15)(15)^*\set{15}^c,\\
\set{2,3,6,7,11}1^*\set{1}^c(15)^*\set{15}^c,\\
\set{4,5,8,9}(15)^*\set{15}^c,\\
(10)\set{1,2,3,4,5,6,7,15}^c,\\
(10)11^*(15)(15)^*\set{15}^c,\\
(10)\set{2,3,4,5,6,7,15}(15)^*\set{15}^c,\\
(12)\set{13,14},\\
(12)11^*\set{1}^c(15)^*\set{15}^c,\\
(12)\set{2,3,4,5,8,9,10,11,12}(15)^*\set{15}^c,\\
(12)\set{6,7}1^*\set{1,15}^c,\\
(12)\set{6,7}1^*(15)(15)^*\set{15}^c,\\
(13)11^*\set{1}^c(15)^*\set{15}^c,\\
(13)\set{2,3,4,5,6,7,8,9,10,11,14,15}(15)^*\set{15}^c,\\
(13)\set{12,13}1^*\set{1,15}^c,\\
(13)\set{12,13}1^*(15)(15)^*\set{15}^c,\\
(14)\set{1,2,3,4,15}^c,\\
(14)\set{1,2,3}1^*\set{1,15}^c,\displaybreak[0]\\
(14)\set{1,2,3}1^*(15)(15)^*\set{15}^c,\\
(14)\set{4,15}(15)^*\set{15}^c.
\end{gather*}
This completes the proof of Theorem~\ref{ref20}.

\end{document}